\documentclass{amsart}
\usepackage[latin1]{inputenc}
\usepackage{a4}
\usepackage{morefloats}
\usepackage{amsfonts,amssymb,amsmath,amsthm}
\usepackage{tikz}
\newcommand{\SIZE}{\operatorname{size}}
\newcommand{\SAST}{\operatorname{s}_\ast}
\newcommand{\STRI}{\operatorname{s}_\text{\scriptsize\rm tri}}
\newcommand{\SPAR}{\operatorname{s}_\text{\scriptsize\rm par}}
\newcommand{\STRA}{\operatorname{s}_\text{\scriptsize\rm trap}}
\newcommand{\SPENT}{\operatorname{s}_\text{\scriptsize\rm pent}}
\newcommand{\SHEX}{\operatorname{s}_\text{\scriptsize\rm hex}}
\newcommand{\SASTp}{\operatorname{s}_\ast^\text{\scriptsize\rm t-perf}}
\newcommand{\STRIp}{\operatorname{s}_\text{\scriptsize\rm tri}^\text{\scriptsize\rm t-perf}}
\newcommand{\SPARp}{\operatorname{s}_\text{\scriptsize\rm par}^\text{\scriptsize\rm t-perf}}
\newcommand{\STRAp}{\operatorname{s}_\text{\scriptsize\rm trap}^\text{\scriptsize\rm t-perf}}
\newcommand{\SPENTp}{\operatorname{s}_\text{\scriptsize\rm pent}^\text{\scriptsize\rm t-perf}}
\newcommand{\SHEXp}{\operatorname{s}_\text{\scriptsize\rm hex}^\text{\scriptsize\rm t-perf}}
\newcommand{\DOM}{\operatorname{dom}}
\newcommand{\LENGTH}{\operatorname{length}}
\newtheorem{theorem}{Theorem}
\newtheorem{corollary}[theorem]{Corollary}
\newtheorem{lemma}[theorem]{Lemma}

\begin{document}

%%%%%%%%%%%%%%%%%%%%%%%%%%%%%%%%%%%%%%%%%%%%%%%%%%

\title[Tilings of convex polygons by equilateral triangles]{Tilings of convex polygons by equilateral triangles of many different sizes}

\author{Christian Richter}
\address{Institute of Mathematics, Friedrich Schiller University,
D-07737 Jena, Germany}
\email{christian.richter@uni-jena.de}

\date{\today}

\begin{abstract}
An equilateral triangle cannot be dissected into finitely many mutually incongruent equilateral triangles \cite{tutte1948}. Therefore Tuza \cite{tuza1991} asked for the largest number $s=s(n)$ such that there is a tiling of an equilateral triangle by $n$ equilateral triangles of $s(n)$ different sizes. We solve that problem completely and consider the analogous questions for dissections of convex $k$-gons into equilateral triangles, $k=4,5,6$. Moreover, we discuss all these questions for the subclass of tilings such that no two tiles are translates of each other.
\end{abstract}

\maketitle

%%%%%%%%%%%%%%%%%%%%%%%%%%%%%%%%%%%%%%%%%%%%%%%%%%

\section{Overview}

\subsection{Motivation}

A (finite) \emph{tiling} of a subset $A$ of the Euclidean plane $\mathbb{R}^2$ is a family $\mathcal{T}=\{A_1,\ldots,A_n\}$ of subsets of $A$, called \emph{tiles}, such that $A=A_1 \cup \ldots \cup A_n$ and the interiors of $A_1,\ldots,A_n$ are mutually disjoint. A tiling is called \emph{perfect} if all tiles are images of each other under similarity transformations, but mutually incongruent under isometries. The concept of a perfect tiling has been introduced in \cite{brooks1940} for the construction of perfect dissections of rectangles and even of squares into squares. Tutte showed that there are no perfect tilings of triangles by (at least two) equilateral triangles \cite[Theorem 2$\cdot$12]{tutte1948} (see also \cite[Section 10.3]{brooks1940}). In fact, no convex polygon admits a perfect tiling by equilateral triangles, as has been shown by Buchman \cite[Section 2]{buchman1981} and Tuza \cite[Theorem~1]{tuza1991}. 

Therefore Tuza considered tilings of triangles by equilateral triangles that are nearly perfect in the sense that only few of the tiles are isometrically congruent: The \emph{size} $\SIZE(T)$ of an equilateral triangle $T$ is measured by the length of (one of) its sides. For a family $\mathcal{T}=\{T_1,\ldots,T_n\}$ of equilateral triangles, we define $s(\mathcal{T})=|\{\SIZE(T_1),\ldots,\SIZE(T_n)\}|$, where $|\cdot|$ denotes the cardinality of a set. Tuza introduced the numbers
\[
\STRI(n)=\max\{s(\mathcal{T}): \mathcal{T} \mbox{ is a tiling of a triangle by } n \mbox{ equilateral triangles}\}.
\]
He showed by an inflation argument that there exists a constant $c$, $\frac{5}{7} \le c \le 1$, such that $\STRI(n)=cn-o(n)$ as $n \to \infty$ \cite[Theorem~2]{tuza1991} and he asked for the exact values of $\STRI(n)$ \cite[Problem~1]{tuza1991} or at least for the exact value of $c$. We solve that problem completely in Theorem~\ref{thm:general}(a).

If a convex polygon $P$ admits a tiling by equilateral triangles, its inner angles are of size $\frac{\pi}{3}$ or $\frac{2\pi}{3}$ and $P$ must be an equilateral triangle (all inner angles have size $\frac{\pi}{3})$, a trapezoid (sizes of angles are $\frac{\pi}{3}, \frac{\pi}{3}, \frac{2\pi}{3}, \frac{2\pi}{3}$ in cyclic order), a parallelogram ($\frac{\pi}{3}, \frac{2\pi}{3}, \frac{\pi}{3}, \frac{2\pi}{3}$), a pentagon ($\frac{\pi}{3}, \frac{2\pi}{3}, \frac{2\pi}{3}, \frac{2\pi}{3}, \frac{2\pi}{3}$) or a hexagon (six times $\frac{2\pi}{3}$). \emph{(When speaking of trapezoids in the present paper we mean trapezoids that are no parallelograms.)} As we know that no convex polygon has a perfect tiling by equilateral triangles, we introduce relatives of Tuza's numbers, namely
\begin{align*}
\STRA(n)=\max\{s(\mathcal{T}):\;& \mathcal{T} \mbox{ is a tiling of some convex}\\
&\mbox{trapezoid by } n \mbox{ equilateral triangles}\}
\end{align*}
and, similarly, $\SPAR(n)$, $\SPENT(n)$ and $\SHEX(n)$ with `trapezoid' replaced by `parallelogram', `pentagon' and `hexagon', respectively. Theorem~\ref{thm:general} gives our results on these numbers including the determination of the \emph{domains} of the functions $\STRI,\ldots,\SHEX$. These are the sets of all integers $n$ such that there exists a tiling of a suitable convex polygon of the respective shape by $n$ equilateral triangles.

Another possibility of weakening the property of perfectness is based on the fact that, if two tiles of a tiling of a convex set by equilateral triangles are congruent under isometries, then they are either congruent under some translation or under some rotation by an angle of $\pi$. Some authors call the tiling already perfect if no two tiles are congruent under translations \cite{tutte1948,drapal2010}. In order to avoid confusion with (isometric) perfectness, we shall speak of translational perfectness (or t-perfectness for short). That is, a tiling by equilateral triangles is called \emph{t-perfect} if no two tiles are translates of each other. First examples of t-perfect tilings of parallelograms and triangles by equilateral triangles are given in \cite{tutte1948}. Dr\'apal and H\"am\"al\"ainen \cite{drapal2010} present a systematic computational approach to dissections of equilateral triangles into equilateral triangles with a particular emphasis on the t-perfect case. Illustrations of all t-perfect tilings of triangles by up to $19$ equilateral triangles are given in \cite{hamalainen}.

We are interested in t-perfect tilings of convex polygons by equilateral triangles that are close to be (isometrically) perfect in so far as the number of tiles of equal size is as small as possible. More precisely, we study the numbers
\begin{align*}
\STRIp(n)=\max\{s(\mathcal{T}):\,& \mathcal{T} \mbox{ is a t-perfect tiling of a}\\
&\mbox{triangle by } n \mbox{ equilateral triangles}\}
\end{align*}
as well as the analogous variants $\STRAp(n)$, $\SPARp(n)$, $\SPENTp(n)$ and $\SHEXp(n)$ of $\STRA(n)$, $\SPAR(n)$, $\SPENT(n)$ and $\SHEX(n)$, respectively. Our corresponding results are summarized in Theorem~\ref{thm:perfect}.

Finally, let us point out that the study of dissections into incongruent equilateral triangles is a fruitful field of ongoing research, see e.g.\ \cite[Section C11]{croft1991}, \cite[Exercise 2.4.10]{gruenbaum1987}, \cite[Problem~4]{nandakumar} and \cite{aduddell2017, klaassen1995, pach2018, richter2012, richter2018, richter-wirth, scherer1983}.

%%%

\subsection{Main results and open problems}

\begin{theorem}\label{thm:general}
\renewcommand{\labelenumi}{\bf (\alph{enumi})}
\begin{enumerate}
\item 
$\DOM(\STRI)=\{1,4\} \cup \{6,7,\ldots\}$ and
$$
\STRI(n)=\left\{
\begin{array}{cl}
1, & n=1,4,\\
2, & n=6,\\
n-5, &n=7,8,\ldots
\end{array}
\right.
$$

\item
$\DOM(\STRA)=\{3\} \cup \{5,6,\ldots\}$ and
$$
\STRA(n)=\left\{
\begin{array}{cl}
1, & n=3,\\
2, & n=5,\\
n-4, &n=6,7,\ldots
\end{array}
\right.
$$

\item
$\DOM(\SPAR)=\{2\} \cup \{4,5,\ldots\}$ and
$$
\SPAR(n)=\left\{
\begin{array}{cl}
1, & n=2,4,\\
2, & n=5,\\
n-4, &n=6,7,\ldots
\end{array}
\right.
$$

\item
$\DOM(\SPENT)=\{4,5,\ldots\}$ and
$$
\SPENT(n)=\left\{
\begin{array}{cl}
2, & n=4,\\
n-3, &n=5,6,\ldots
\end{array}
\right.
$$

\item
$\DOM(\SHEX)=\{6,7,\ldots\}$ and
$$
\SHEX(n)\left\{
\begin{array}{ll}
=n-5, & n=6,7,8,\\
=n-4, & n=9,10,\ldots,19;\;21,22;\;24,25,\\
\in \{n-5,n-4\}, & n=20;\;23;\;26,27,\ldots
\end{array}
\right.
$$
\end{enumerate}
\end{theorem}

Here only the case of hexagons is not completely solved. All tilings we know to attain the upper bound $\SHEX(n) \le n-4$ ($n \ge 9$) are displayed in the appendix (tilings (a)-(s)). Is this list complete?

The situation of t-perfect tilings appears more difficult. Only for pentagons we have a full solution. However, in the open cases the differences between upper and lower estimates are at most $2$. In the case of hexagons it remains open if there are t-perfect tilings by $12$ or $13$ triangles. We conjecture that $12,13 \notin\DOM\left(\SHEXp\right)$.

\begin{theorem}\label{thm:perfect}
\renewcommand{\labelenumi}{\bf (\alph{enumi})}
\begin{enumerate}
\item 
$\DOM\left(\STRIp\right)=\{1\}\cup\{15,16,\ldots\}$ and
$$\STRIp(n)\left\{
\begin{array}{ll}
=1, & n=1,\\
=n-5, & n=15;\;17,18,\ldots,26;\;28,\\
=n-6, & n=16,\\
\in \{n-6,n-5\}, &n=27;\;29,30,\ldots
\end{array}
\right.
$$

\item
$\DOM\left(\STRAp\right)= \{13,14,\ldots\}$ and
$$
\STRAp(n)\left\{
\begin{array}{ll}
=n-4, & n=14;\;16,17,\ldots,25;\;27,\\
\in\{n-5,n-4\}, &n=13;\;15;\;26;\;28,29,\ldots
\end{array}
\right.
$$

\item
$\DOM\left(\SPARp\right)= \{2\} \cup \{13,14,\ldots\}$ and
$$
\SPARp(n)\left\{
\begin{array}{ll}
=1, & n=2,\\
=n-4, &n=15;\;18,19;\;21,22,23;\;26,\\
\in\{n-5,n-4\}, &n=13,14;\;16,17;\;20;\;24,25;\;27,28,\ldots
\end{array}
\right.
$$

\item
$\DOM\left(\SPENTp\right)= \{12,13,\ldots\}$ and
$$
\SPENTp(n)=n-4 \mbox{ for all } n=12,13,\ldots
$$

\item
$\{11\} \cup \{14,15,\ldots\} \subseteq \DOM\left(\SHEXp\right) \subseteq \{11,12,\ldots\}$ and $$
\SHEXp(n)\left\{
\begin{array}{ll}
=n-4, & n=11;\;14,15;\;17,18,19;\;22,\\
\in \{n-5,n-4\}, & n=16;\;20,21;\;23,\\
\in \{n-6,n-5,n-4\}, & n=24,25,\ldots
\end{array}
\right.
$$
\end{enumerate}
\end{theorem}

%%%%%%%%%%%%%%%%%%%%%%%%%%%%%%%%%%%%%%%%%%%%%%%%%%

\section{Proof of Theorem~\ref{thm:general}}

\subsection{Spiral pentagons and related tilings}

The \emph{Padovan spiral numbers} \cite{steward1996, oeis} are defined recursively by
\begin{equation}\label{eq:Pad}
p(0)=p(1)=p(2)=1 \quad\mbox{and}\quad p(n)=p(n-3)+p(n-2) \mbox{ for } n=3,4,\ldots
\end{equation}

\begin{lemma}\label{lem:spiral}
\renewcommand{\labelenumi}{\bf (\alph{enumi})}
\begin{enumerate}
\item
The Padovan spiral numbers satisfy 
\begin{itemize}
\item
$p(0)=p(1)=p(2)=1$, $p(3)=p(4)=2$ and $p(n)>p(n-1)$ for $n \ge 5$,
\item
$p(n-3) < \frac{1}{2}p(n) < p(n-2)$ for $n \notin \{3,4,6\}$.
\end{itemize}
\item
For every $n \in \{4,5,\ldots\}$, there is a convex pentagon $P_n$ with sides of lengths $p(n-4)$, $p(n-3)$, $p(n-2)$, $p(n-1)$ and $p(n)$ that admits a tiling by $n$ equilateral triangles $T_i$ with $\SIZE(T_i)=p(i-1)$, $i=1,\ldots,n$ (see \cite{steward1996}).
\end{enumerate}
\end{lemma}

\begin{proof}
The first part of (a) follows from \eqref{eq:Pad} by induction. It implies the second one by $p(n)=p(n-3)+p(n-2)$.
Claim (b) can be found in \cite{steward1996}, see Figure~\ref{fig:spiral}: the pentagon $P_{n+1}$ is obtained from $P_n$ by adding a triangle of size $p(n)$ at the longest side of $P_n$.
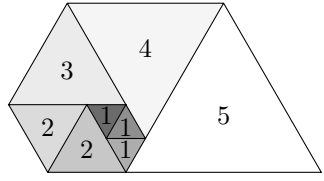
\begin{figure}
\begin{center}
\begin{tikzpicture}[xscale=.1299,yscale=.15]

\fill[black!4]
  (22,15)--(6,15)--(14,3)--cycle
  ;
\fill[black!8]
  (6,15)--(0,6)--(12,6)--cycle
  ;
\fill[black!14]
  (0,6)--(4,0)--(8,6)--cycle
  ;
\fill[black!22]
  (4,0)--(12,0)--(8,6)--cycle
  ;
\fill[black!32]
  (12,0)--(14,3)--(10,3)--cycle
  ;
\fill[black!44]
  (14,3)--(10,3)--(12,6)--cycle
  ;
\fill[black!58]
  (10,3)--(12,6)--(8,6)--cycle
  ;

\draw 
  (0,6)--(4,0)--(32,0)--(22,15)--(6,15)--(0,6)--(12,6)--(10,3)--(14,3)--(6,15)
  (4,0)--(8,6)--(12,0)--(22,15)
  (10,5) node {$1$}
  (12,4) node {$1$}
  (12,2) node {$1$}
  (8,2) node {$2$}
  (4,4) node {$2$}
  (6,9) node {$3$}
  (14,11) node {$4$}
  (22,5) node {$5$}
  ;

\end{tikzpicture}
\end{center}
\caption{The spiral pentagon $P_n$ (here with with $n=8$), cf. \cite{steward1996}}
\label{fig:spiral}
\end{figure}
\end{proof}

\begin{corollary}\label{cor:lower}
\renewcommand{\labelenumi}{\bf (\alph{enumi})}
\begin{enumerate}
\item 
$\{6,7,\ldots\}\subseteq\DOM(\STRI)$ and 
$$
\STRI(n)\ge\left\{
\begin{array}{cl}
2, & n=6,\\
n-5, &n=7,8,\ldots
\end{array}
\right.
$$

\item
$\{5,6,\ldots\}\subseteq\DOM(\STRA)$ and
$$
\STRA(n)\ge\left\{
\begin{array}{cl}
2, & n=5,\\
n-4, &n=6,7,\ldots
\end{array}
\right.
$$

\item
$\{5,6,\ldots\}\subseteq\DOM(\SPAR)$ and
$$
\SPAR(n)\ge\left\{
\begin{array}{cl}
2, & n=5,\\
n-4, &n=6,7,\ldots
\end{array}
\right.
$$

\item
$\{4,5,\ldots\}\subseteq\DOM(\SPENT)$ and
$$
\SPENT(n)\ge\left\{
\begin{array}{cl}
2, & n=4,\\
n-3, &n=5,6,\ldots
\end{array}
\right.
$$

\item
$\{7,8,\ldots\}\subseteq\DOM(\SHEX)$ and $\SHEX(n) \ge n-5$ for all $n=7,8,\ldots$
\end{enumerate}
\end{corollary}

\begin{proof}
The pentagons $P_n$, $n= 4,5,\ldots$, from Lemma~\ref{lem:spiral}(b) prove (d). For (a), we add two triangles of sizes $p(n-4)$ and $p(n-2)$ at the respective sides of $P_n$, cf. Figure~\ref{fig:derived polygons}.
\begin{figure}
\begin{center}
\begin{tikzpicture}

\node at (0,0) {\tikz[xscale=.039,yscale=.045]{
\draw 
  (0,6)--(4,0)--(32,0)--(22,15)--(6,15)--(0,6)
  (0,6)--(-4,0)--(4,0)
  (6,15)--(14,27)--(22,15)
  (15,7.5) node {$P_n$}
  ;
}}
;
\node at (2.5,0) {\tikz[xscale=.039,yscale=.045]{
\draw 
  (0,6)--(4,0)--(32,0)--(22,15)--(6,15)--(0,6)
  (0,6)--(-4,0)--(4,0)
%  (6,15)--(14,27)--(22,15)
  (15,7.5) node {$P_n$}
  ;
}}
;
\node at (5,0) {\tikz[xscale=.039,yscale=.045]{
\draw 
  (0,6)--(4,0)--(32,0)--(22,15)--(6,15)--(0,6)
  (0,6)--(-6,15)--(22,15)
  (15,7.5) node {$P_n$}
  ;
}}
;
\node at (7.5,0) {\tikz[xscale=.039,yscale=.045]{
\draw 
  (0,6)--(4,0)--(32,0)--(22,15)--(6,15)--(0,6)
  (4,0)--(11,-10.5)--(25,-10.5)--(32,0)
  (11,-10.5)--(18,0)--(25,-10.5)
  (15,7.5) node {$P_n$}
  ;
}}
;
\node at (0,-1) {(a)};
\node at (2.5,-1) {(b)};
\node at (5,-1) {(c)};
\node at (7.5,-1) {(e)};
\end{tikzpicture}
\end{center}
\caption{Modifications of $P_n$ by added equilateral triangles}
\label{fig:derived polygons}
\end{figure}
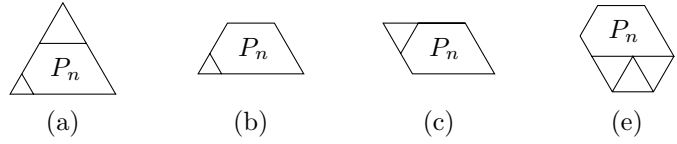
For (b), we add only one of these two triangles. For (c), we add a triangle of size $p(n-3)$. Every added triangle is of the same size as one of the original tiling of $P_n$ from Lemma~\ref{lem:spiral}(b). 

For (e), we add three triangles of size $\frac{1}{2}p(n)$ at the longest side of length $p(n)$ if $n \ne 6$, see Figure~\ref{fig:derived polygons}. By Lemma~\ref{lem:spiral}(a), the size of the new triangles differs from all sizes of the given tiling of $P_n$ if $n=5$ or $n \ge 7$, whereas that size appears already in the tiling of $P_n$ if $n=4$. For $n=6$, we add three triangles of size $\frac{1}{2}p(5)=\frac{3}{2}\notin\{p(0),\ldots,p(5)\}$ at the side of size $p(5)=3$ of $P_6$.
\end{proof}

%%%

\subsection{Domains and lower estimates\label{subsec:realizations}}

If a convex polygon $P$ has a tiling $\mathcal{T}$ by equilateral triangles, its inner angles are of sizes $\frac{\pi}{3}$ and $\frac{2\pi}{3}$. A side between two vertices of sizes $\alpha$ and $\beta$ is called an \emph{$(\alpha,\beta)$-side of $P$}.
All vertices of triangles from $\mathcal{T}$ are called \emph{vertices of $\mathcal{T}$}. We speak of \emph{$\frac{\pi}{3}$-vertices or $\frac{2\pi}{3}$-vertices of $\mathcal{T}$}, if they coincide with vertices of $P$ of the respective sizes, of \emph{$\pi$-vertices of $\mathcal{T}$}, if they are no vertices of $P$ but on the boundary of $P$, and of \emph{$2\pi$-vertices of $\mathcal{T}$} if they are in the interior of $P$. A triangle $T \in \mathcal{T}$ is called \emph{exposed} if it contains a $\frac{\pi}{3}$-vertex of $\mathcal{T}$ or, equivalently, if $P \setminus T$ is still convex.

Let us show that
\begin{eqnarray}
2,3;\; 5 & \notin & \DOM(\STRI),\label{eq:dom3}\\
1,2;\; 4 & \notin & \DOM(\STRA),\label{eq:dom4t}\\
1;\; 3 & \notin & \DOM(\SPAR),\label{eq:dom4p}\\
1,2,3 & \notin & \DOM(\SPENT),\label{eq:dom5}\\
1,2,3,4,5 & \notin & \DOM(\SHEX)\label{eq:dom6}.
\end{eqnarray}

To see \eqref{eq:dom3}, let $n$ be the cardinality of a tiling $\mathcal{T}$ of some triangle $T$ and let $v_\pi$ be the number of $\pi$-vertices of $\mathcal{T}$. If $v_\pi=0$ then $n=1$. If $1 \le v_\pi \le 3$ then $v_\pi=3$ and $\mathcal{T}$ splits into three congruent exposed triangles and a tiling $\mathcal{T}'$ of the remaining triangle $T'$ of $T$. The cardinality $|\mathcal{T}'|=n-3$ cannot be two, whence $n=4$ (if $|\mathcal{T}'|=1$) or $n \ge 6$ (if $|\mathcal{T}'| \ge 3$). If $v_\pi \ge 4$ then one side of $T$ contains two consecutive $\pi$-vertices $v_1,v_2$. Hence $n \ge 6$, because $\mathcal{T}$ contains three exposed triangles, one triangle covering the segment $v_1v_2$ and two triangles for covering the yet uncovered remainders of neighbourhoods of $v_1$ and $v_2$.

For \eqref{eq:dom4t}, note that $n \in \DOM(\STRA)$ implies $n+1 \in \DOM(\STRI)$, since a trapezoid can be transformed into a triangle by adding one triangle at its $\left(\frac{2\pi}{3},\frac{2\pi}{3}\right)$-side. Hence \eqref{eq:dom3} implies \eqref{eq:dom4t}. Similarly, $n \in \DOM(\SPAR)$ implies $n+1 \in \DOM(\STRA)$ by adding a suitable triangle. Thus $\eqref{eq:dom4p}$ is a consequence of $\eqref{eq:dom4t}$.

A tiling of a convex pentagon contains at least four tiles, because it must contain one exposed triangle at the $\frac{\pi}{3}$-vertex and at least one tile for each of the three $\left(\frac{2\pi}{3},\frac{2\pi}{3}\right)$-sides. This gives \eqref{eq:dom5}. Similarly, we obtain \eqref{eq:dom6}, since a tiling of a convex hexagon contains at least one tile for each of the six sides of the hexagon. 

Claims \eqref{eq:dom3}-\eqref{eq:dom6} together with Figure~\ref{fig:small tilings} and Corollary~\ref{cor:lower} show that the domains of $\STRI$, $\STRA$, $\SPAR$, $\SPENT$ and $\SHEX$ are as claimed in Theorem~\ref{thm:general}.
\begin{figure}
\begin{center}
\begin{tikzpicture}[xscale=.35,yscale=.3031]

\draw 
  (1.5,1)--(3.5,1)--(2.5,3)--(1.5,1)
  (4,0)--(8,0)--(6,4)--(4,0) (6,0)--(7,2)--(5,2)--(6,0)
  
  (11,3)--(10,1)--(14,1)--(13,3)--(11,3)--(12,1)--(13,3)
  
  (17,3)--(16,1)--(18,1)--(19,3)--(17,3)--(18,1)
  (20,3)--(19,1)--(23,1)--(24,3)--(20,3)--(21,1)--(22,3)--(23,1)
  
  (26,2)--(27,0)--(29,0)--(30,2)--(29,4)--(27,4)--(26,2)--(30,2) (27,0)--(29,4) (29,0)--(27,4)
  ;
\end{tikzpicture}
\end{center}
\caption{Tilings of small cardinality and with congruent tiles}
\label{fig:small tilings}
\end{figure}
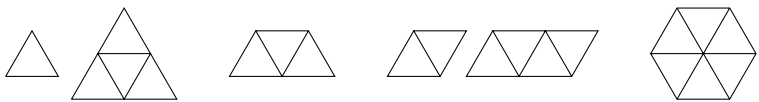
Moreover, Corollary~\ref{cor:lower} gives the lower estimates of these functions from Theorem~\ref{thm:general} apart from $\SHEX(n) \ge n-4$ for $n \in \{9,\ldots,19\} \cup \{21,22,24,25\}$. These last estimates are justified by the tilings (a)-(s) presented in the appendix. We do not know other tilings of hexagons that realize $\SHEX(n) \ge n-4$. Table~\ref{tab:hexagons} gives the parameters of the illustrated tilings.

%%%

\subsection{A necessary condition}

\begin{lemma}\label{lem:necessary}
Let $\mathcal{T}$ be a tiling of a convex $m$-gon $P$ by at least two equilateral triangles, and let $v_\pi$ be the number of $\pi$-vertices of $\mathcal{T}$. Then there are at least $m+v_\pi-3$ pairs of distinct triangles from $\mathcal{T}$ having a side in common.
\end{lemma}

\begin{proof}
We can assume that
\begin{equation}\label{eq:1}
v_\pi \ge 9-2m.
\end{equation}
Indeed, this is trivial if $m \ge 5$. If $m=3$ then $v_\pi \ge 3$, since the trivial tiling $\mathcal{T}=\{P\}$ is excluded. If $m=4$ and $v_\pi=0$ then $\mathcal{T}$ is a tiling of a rhombus $P$ by two equilateral triangles, and the claim of the lemma is obvious.

The following graph-theoretic arguments generalize a similar approach from \cite{tutte1948, buchman1981}. We associate a bipartite planar graph $\Gamma$ to $\mathcal{T}$: In the interior of every triangle from $\mathcal{T}$ we place a white node of $\Gamma$. All vertices of $\mathcal{T}$, except for the $\frac{\pi}{3}$-vertices, are the black nodes of $\Gamma$. A black and a white node of $\Gamma$ are joined by an edge if the black one is a vertex of the triangle represented by the white one, see Figure~\ref{fig:graph}. \begin{figure}
\begin{center}
\begin{tikzpicture}[xscale=.5,yscale=.5]

\draw[densely dotted]
  (2,3.464)--(0,0)
  (0,0)--(6,0)
  (6,0)--(8,3.464)
  (8,3.464)--(2,3.464)--(4,0)--(6,3.464)--(7,1.732)--(3,1.732)--(4,3.464)--(6,0)
  ;
  
\draw 
  (2,1.155) circle (.2)
  (2,3.464)--(2,1.355) 
  (4,0)--(2.173,1.055) 
  (4,1.155) circle (.2)
  (4,.955)--(4,0)
  (3.827,1.255)--(3,1.732)
  (4.173,1.255)--(5,1.732)
  (6,1.155) circle (.2)
  (6,.955)--(6,0)
  (5.827,1.255)--(5,1.732)
  (6.173,1.255)--(7,1.732)
  (5,.577) circle (.2)
  (5,.777)--(5,1.732)
  (4.827,.477)--(4,0)
  (5.173,.477)--(6,0)
  (3,2.887) circle (.2)
  (3,2.687)--(3,1.732)
  (2.827,2.987)--(2,3.464)
  (3.173,2.987)--(4,3.464)
  (5,2.887) circle (.2)
  (5,2.687)--(5,1.732)
  (4.827,2.987)--(4,3.464)
  (5.173,2.987)--(6,3.464)
  (7,2.887) circle (.2)
  (7,2.687)--(7,1.732)
  (6.827,2.987)--(6,3.464)
  (4,2.309) circle (.2)
  (4,2.509)--(4,3.464)
  (3.827,2.209)--(3,1.732)
  (4.173,2.209)--(5,1.732)
  (6,2.309) circle (.2)
  (6,2.509)--(6,3.464)
  (5.827,2.209)--(5,1.732)
  (6.173,2.209)--(7,1.732)
  ;

\fill
  (2,3.464) circle (.2)
  (4,3.464) circle (.2)
  (6,3.464) circle (.2)
  (3,1.732) circle (.2)
  (5,1.732) circle (.2)
  (7,1.732) circle (.2)
  (4,0) circle (.2)
  (6,0) circle (.2)
  ;

\end{tikzpicture}
\end{center}
\caption{A tiling of a parallelogram (dotted) and the associated graph}
\label{fig:graph}
\end{figure}
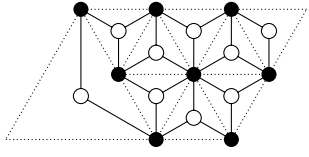
Let $v$, $e$ and $f$ be the numbers of nodes, edges and faces of $\Gamma$, respectively.

All nodes of $\Gamma$ have degree $2$, $3$ or $6$. Denoting the respective numbers by $v_2$, $v_3$ and $v_6$, we have
\begin{equation}\label{eq:2}
v=v_2+v_3+v_6.
\end{equation}
Since $P$ is a convex $m$-gon whose inner angles have sizes of $\frac{\pi}{3}$ or $\frac{2\pi}{3}$, it has $6-m$ inner angles of size $\frac{\pi}{3}$ and $2m-6$ inner angles of size $\frac{2\pi}{3}$. A white node of $\Gamma$ has degree $2$ if and only if it represents a triangle that covers an angle of $P$ of size $\frac{\pi}{3}$. A black node of $\Gamma$ has degree $2$ if and only if it corresponds to an angle of $P$ of size $\frac{2\pi}{3}$. Thus
\begin{equation}\label{eq:3}
v_2=(6-m)+(2m-6)=m.
\end{equation}
Counting the edges of $\Gamma$ in terms of the nodes we obtain
\begin{equation}\label{eq:4}
2e=2v_2+3v_3+6v_6.
\end{equation}

Since $\Gamma$ is bipartite, we have
\begin{equation}\label{eq:5}
f=\sum_{i=2}^\infty f_{2i}
\end{equation}
where $f_j$ is the number of faces with $j$ edges. The boundary of $P$ contains exactly $(2m-6)+v_\pi$ black nodes of $\Gamma$: $2m-6$ of them represent inner angles of size $\frac{2\pi}{3}$ and $v_\pi$ of them are $\pi$-vertices. Hence the unbounded face of $\Gamma$ has $2((2m-6)+v_\pi)$ edges, and in turn
\begin{equation}\label{eq:6}
f_{2(2m-6+v_\pi)} \ge 1.
\end{equation}
Counting the edges of $\Gamma$ in terms of the faces we obtain
\begin{equation}\label{eq:7}
2e=\sum_{i=2}^\infty 2if_{2i}.
\end{equation}

Euler's formula for $\Gamma$ gives
\begin{eqnarray*}
2 &=& f-e+v\\
&=& \left(f-\frac{1}{6}\,2e\right)+\left(v-\frac{1}{3}\,2e\right)\\
&\stackrel{\text{(\ref{eq:2},\ref{eq:4},\ref{eq:5},\ref{eq:7})}}{=}& \left(\sum_{i=2}^\infty f_{2i}-\frac{1}{6}\sum_{i=2}^\infty 2if_{2i}\right)+\left(v_2+v_3+v_6-\frac{1}{3}(2v_2+3v_3+6v_6)\right)\\
&\stackrel{\text{(\ref{eq:3})}}{=}& \frac{1}{3}\left(f_4-\sum_{i=3}^\infty (i-3)f_{2i}\right)+\frac{1}{3}\left(m-3v_6\right).
\end{eqnarray*}
Thus
$$
f_4=6+\sum_{i=3}^\infty (i-3)f_{2i}-m+3v_6.
$$
Using $2m-6+v_\pi \ge 3$ (as a consequence of \eqref{eq:1}) and $v_6 \ge 0$ we get
$$
f_4 \ge 6+ ((2m-6+v_\pi)-3)f_{2(2m-6+v_\pi)}-m
$$
and, by \eqref{eq:6},
$$
f_4 \ge 6+ ((2m-6+v_\pi)-3)-m=m+v_\pi-3.
$$
This estimate completes the proof, because every face of $\Gamma$ with four edges represents two triangles of $\mathcal{T}$ having a side in common.
\end{proof}

%%%

\subsection{Upper estimates\label{subsec:upgen}}

We shall use the following observation.

\begin{lemma}\label{lem:trapezoid}
If a tiling $\mathcal{T}$ of a trapezoid by equilateral triangles satisfies one of the conditions
\begin{itemize}
\item[($\alpha$)]
all $\pi$-vertices of $\mathcal{T}$ are contained in the $\left(\frac{\pi}{3},\frac{\pi}{3}\right)$-side of the trapezoid or
\item[($\beta$)]
$|\mathcal{T}|=3$,
\end{itemize}
then $\mathcal{T}$ is a similar image of the respective tiling from Figure~\ref{fig:small tilings} and $s(\mathcal{T})=1$.
\end{lemma}

\begin{proof}
We can assume ($\alpha$), since ($\beta$) implies ($\alpha$).
Let $T_1,T_2 \in \mathcal{T}$ be the tiles covering the two $\left( \frac{\pi}{3},\frac{2\pi}{3} \right)$-sides of the trapezoid and let $T_3 \in \mathcal{T}$ be the triangle that covers the $\left( \frac{2\pi}{3},\frac{2\pi}{3} \right)$-side. Since the two $\left( \frac{\pi}{3},\frac{2\pi}{3} \right)$-sides have the same length, $\SIZE(T_1)=\SIZE(T_2)$. We can exclude the situations $\SIZE(T_3) < \SIZE(T_1)$, because then the interiors of $T_1$ and $T_2$ would overlap, and $\SIZE(T_3) > \SIZE(T_1)$, since then the third vertex of $T_3$ would be outside the trapezoid. Hence $\SIZE(T_1)=\SIZE(T_2)=\SIZE(T_3)$ and $T_3$ covers the remainder of the trapezoid between $T_1$ and $T_2$. The claim follows.
\end{proof}

For completing the proof of Theorem~\ref{thm:general}, it remains to show that
\begin{eqnarray}
\STRI(n) &\le& \left\{
\begin{array}{cl}
1, & n=1,4,\\
2, & n=6,\\
n-5, &n=7,8,\ldots,
\end{array}
\right.
\label{eq:uptri}\\
\STRA(n) &\le& \left\{
\begin{array}{cl}
1, & n=3,\\
2, & n=5,\\
n-4, &n=6,7,\ldots,
\end{array}
\right.
\label{eq:uptra}\\
\SPAR(n) &\le& \left\{
\begin{array}{cl}
1, & n=2,4,\\
2, & n=5,\\
n-4, &n=6,7,\ldots,
\end{array}
\right.
\label{eq:uppar}\\
\SPENT(n) &\le& \left\{
\begin{array}{cl}
2, & n=4,\\
n-3, &n=5,6,\ldots,
\end{array}
\right.
\label{eq:uppen}\\
\SHEX(n) &\le& \left\{
\begin{array}{cl}
n-5, & n=6,7,8,\\
n-4, &n=9,10,\ldots
\end{array}
\right.
\label{eq:uphex}
\end{eqnarray}
In all situations we consider a tiling $\mathcal{T}$ of a respective polygon $P$. The cardinality and the number of $\pi$-vertices of $\mathcal{T}$ are denoted by $n=|\mathcal{T}|$ and $v_\pi$, respectively. We have to show that $s(\mathcal{T})$ is bounded from above by the claimed upper estimate of $\STRI(n)$, $\STRA(n)$, $\SPAR(n)$, $\SPENT(n)$ or $\SHEX(n)$, respectively.

\begin{proof}[Proof of \eqref{eq:uptri}]
We proceed by induction over $n$. The claim is trivial if $n=1$. If $n=4$ then $\mathcal{T}$ contains three exposed triangles and the fourth tile covers the remainder of $P$. We obtain the respective tiling from Figure~\ref{fig:small tilings} by four congruent tiles. This yields $s(\mathcal{T})=1$. Now let $n \ge 6$.

\emph{Case 1: $v_\pi \le 3$. } Then $v_\pi=3$ and $\mathcal{T}$ splits into three exposed triangles of the same size and a tiling $\mathcal{T}'$ of the remaining equilateral triangle $P'$ of $P$, $P'$ being of that size as well. Now $|\mathcal{T}'|=n-3\ge 3$, and the induction hypothesis gives $s(\mathcal{T}') \le |\mathcal{T}'|-3=n-6$. Since the tiles in $P'$ are smaller than the exposed tiles of $\mathcal{T}$, we obtain $s(\mathcal{T}) = s(\mathcal{T}')+1 \le (n-6)+1=n-5$.

\emph{Case 2: $v_\pi=4$ and $n \ne 6$. } Now $\mathcal{T}$ contains three exposed triangles $T_1$, $T_2$ and $T_3$ such that $T_1$ shares a vertex with each of $T_2$ and $T_3$. Thus $\SIZE(T_2)=\SIZE(T_3)$, but $T_2$ and $T_3$ are disjoint. By Lemma~\ref{lem:necessary}, $\mathcal{T}$ contains at least $3+v_\pi-3=4$ pairs of triangles that share a side. We introduce a graph $\Gamma$ whose nodes are the triangles of $\mathcal{T}$. Two triangles form an edge if these are $T_2$ and $T_3$ or if they have a side in common. This graph has at least $1+4=5$ edges and does not contain a cycle of size less than six. (Cycles not involving the edge $\{T_2,T_3\}$ have size at least six. If a cycle contains $\{T_2,T_3\}$ then all triangles of $\mathcal{T} \setminus \{T_1\}$ are of the same size. By $v_\pi=4$, the side of $P$ that meets both $T_2$ and $T_3$ contains exactly two $\pi$-vertices. Hence there are exactly five triangles in $\mathcal{T} \setminus \{T_1\}$ and in turn $n=|\mathcal{T}|=6$, which is excluded.) We pick five edges of $\Gamma$. Since they do not contain a cycle and since each of them connects two congruent triangles, we have $s(\mathcal{T}) \le |\mathcal{T}|-5=n-5$.

\emph{Case 3: $v_\pi=4$ and $n=6$. } Lemma~\ref{lem:necessary} shows that $\mathcal{T}$ contains at least $3+v_\pi-3=4$ pairs of triangles having a side in common. This yields
$s(\mathcal{T}) \le |\mathcal{T}|-4=2$.  

\emph{Case 4: $v_\pi \ge 5$. } Now Lemma~\ref{lem:necessary} shows that $\mathcal{T}$ contains at least $3+v_\pi-3 \ge 5$ pairs of triangles having a side in common. This yields
$s(\mathcal{T}) \le |\mathcal{T}|-5=n-5$.
\end{proof}

\begin{proof}[Proof of \eqref{eq:uptra}]
We add an equilateral triangle at the $\left(\frac{2\pi}{3},\frac{2\pi}{3}\right)$-side of $P$ and obtain a triangle $P'$ and a corresponding tiling $\mathcal{T}'$ of $P'$ with $|\mathcal{T}'|=n+1$. Then
$$
s(\mathcal{T}) \le s(\mathcal{T}') \le \STRI(n+1) \stackrel{\text{\eqref{eq:uptri}}}{\le}
\left\{
\begin{array}{cl}
1, & n=3,\\
2, & n=5,\\
n-4, &n=6,7,\ldots
\end{array}
\right.
$$
\end{proof}

\begin{proof}[Proof of \eqref{eq:uppar}]
W.l.o.g. $n \ge 4$, since $n \in \DOM(\SPAR)=\{2\} \cup \{4,5,\ldots\}$ and the situation is clear for $n=2$.

\emph{Case 1: One side length of $P$ is larger than $\max\{\SIZE(T): T \in \mathcal{T}\}$. } We add two triangles at two sides of $P$ that meet at a $\frac{\pi}{3}$-vertex of $P$. This way we obtain a tiling $\mathcal{T}'$ of a triangle $P'$ with $|\mathcal{T}'|=n+2$ and $s(\mathcal{T}') \ge s(\mathcal{T})+1$, since one of the new triangles has a new size. Consequently,
$$
s(\mathcal{T}) \le s(\mathcal{T}')-1 \le \STRI(n+2)-1 \stackrel{\text{\eqref{eq:uptri}}}{\le}\left\{
\begin{array}{ll}
2-1=1, & n=4,\\
2-1 < 2, & n=5,\\
((n+2)-5)-1=n-4, &n=6,7,\ldots
\end{array}
\right.
$$

\emph{Case 2: No side length of $P$ is larger than $\max\{\SIZE(T): T \in \mathcal{T}\}$. } Since no side of $P$ can be shorter than $\max\{\SIZE(T): T \in \mathcal{T}\}$, $P$ is a rhombus and one tile $T_0 \in \mathcal{P}$ represents half of $P$. Then $\mathcal{T}'=\mathcal{T}\setminus\{T_0\}$ is a tiling of the other half $P'$ of $P$. In particular, $P'$ is a triangle, $s(\mathcal{T}')\ge s(\mathcal{T})-1$ and 
$$
n \in \{5\} \cup \{7,8,\ldots\},
$$
because $n \in \DOM(\SPAR) \setminus \{2\}=\{4,5,\ldots\} $ and $n-1= |\mathcal{T}'| \in \DOM(\STRI)=\{1,4\}\cup\{6,7,\ldots\}$. We obtain
$$
s(\mathcal{T}) \le s(\mathcal{T}')+1 \le \STRI(n-1)+1 \stackrel{\text{\eqref{eq:uptri}}}{\le}\left\{
\begin{array}{ll}
1+1=2, & n=5,\\
2+1=n-4, & n=7,\\
((n-1)-5)+1< n-4, &n=8,9,\ldots
\end{array}
\right.
$$
\end{proof}

\begin{proof}[Proof of \eqref{eq:uppen}]
\emph{Case 1: $v_\pi=0$. } After removing the exposed triangle from $\mathcal{T}$ as well as from $P$, we obtain a tiling $\mathcal{T}'$ of a trapezoid $P'$ such that all $\pi$-vertices of $\mathcal{T}'$ are on the $\left(\frac{\pi}{3},\frac{\pi}{3}\right)$-side of $P'$. By Lemma~\ref{lem:trapezoid}, $\mathcal{T}'$ is a similar image of the tiling of cardinality three from Figure~\ref{fig:small tilings}. Hence $\mathcal{T}$ consists of $n=4$ tiles, three of them being of the same size. So we have $s(\mathcal{T}) \le 2$, the required upper estimate for $n=4$.

\emph{Case 2: $v_\pi\ge 1$. } By Lemma~\ref{lem:necessary}, $\mathcal{T}$ contains at least $5+v_\pi-3 \ge 3$ pairs of triangles that have a side in common. This yields $s(\mathcal{T}) \le |\mathcal{T}|-3=n-3$.
\end{proof}

\begin{proof}[Proof of \eqref{eq:uphex}]
\emph{Case 1: $v_\pi=0$ or $n \le 6$. } We have $v_\pi=0$. Indeed, if $n \le 6$ then $n=6$ by \eqref{eq:dom6}. Since no tile from $\mathcal{T}$ covers segments of more than one side of the hexagon $P$, each side of $P$ is a side of some triangle from $\mathcal{T}$ and in turn $v_\pi=0$.

Let $T_1,\ldots,T_6 \in \mathcal{T}$ be the tiles that cover the sides of $P$ in successive order. It is enough to show that $\SIZE(T_1)=\ldots=\SIZE(T_6)$, since then $P$ is a regular hexagon of side length $\SIZE(T_1)$ that is completely tiled by $T_1,\ldots,T_6$, hence $n=6$ and $s(\mathcal{T})=1$, as claimed in \eqref{eq:uphex}.

To obtain a contradiction, suppose that, say, $\SIZE(T_6) > \SIZE(T_1)$. Then 
\begin{equation}\label{eq:hex1.1}
\SIZE(T_6) > \SIZE(T_1) \ge \SIZE(T_2) \ge \SIZE(T_3),
\end{equation}
because tiles do not overlap. The side lengths of $P$ satisfy
\begin{equation}\label{eq:hex1.2}
\SIZE(T_2)+\SIZE(T_3)=\SIZE(T_5)+\SIZE(T_6).
\end{equation}
Then \eqref{eq:hex1.1} shows that $\SIZE(T_6)>\SIZE(T_5)$, which implies
\begin{equation}\label{eq:hex1.3}
\SIZE(T_6) > \SIZE(T_5) \ge \SIZE(T_4) \ge \SIZE(T_3)
\end{equation}
as above. Adding the inequalities $\SIZE(T_2) < \SIZE(T_6)$ from \eqref{eq:hex1.1} and 
$\SIZE(T_3) \le \SIZE(T_5)$ from \eqref{eq:hex1.3} gives a contradiction to \eqref{eq:hex1.2}.

\emph{Case 2: $v_\pi=1$ and $n\in\{7,8\}$. } Let $v_0$ be the $\pi$-vertex, let $T_1,\ldots,T_7 \in \mathcal{T}$ be the tiles that form the boundary of $P$ ordered consecutively such that $v_0$ is a common vertex of $T_1$ and $T_7$, and let $T^\ast \in \mathcal{T}$ be the third tile having $v_0$ as a vertex, see Figure~\ref{fig:case2}(a). 
\begin{figure}
\begin{center}
\begin{tikzpicture}

\node at (0,-.22) {
%\begin{tikzpicture}[xscale=.035,yscale=.03031]
\begin{tikzpicture}[xscale=.03,yscale=.02598]

\draw
  (0,30)--(15,0)--(25,20) (15,0)--(45,0)--(35,20) (55,20)--(45,0)--(75,0)--(65,20) (75,0)--(90,30)--(70,30) (90,30)--(75,60)--(65,40) (75,60)--(15,60)--(25,40) (15,60)--(0,30)--(20,30) 
  (45,0) node[below] {$v_0$}
  (60,10) node {$T_1$}
  (75,20) node {$T_2$}
  (75,40) node {$T_3$}
  (45,50) node {$T_4$}
  (15,40) node {$T_5$}
  (15,20) node {$T_6$}
  (30,10) node {$T_7$}
  (45,19) node {$T^\ast$}
  (0,55) node {(a)}
  ;

\end{tikzpicture}
};

\node at (4.33,0) {
%\begin{tikzpicture}[xscale=.035,yscale=.03031]
\begin{tikzpicture}[xscale=.03,yscale=.02598]

\draw
  (15,60)--(0,30)--(15,0)--(45,0)--(75,0)--(90,30)--(75,60)--(15,60)--(45,0)--(75,60) (0,30)--(30,30)--(15,0) (75,0)--(60,30)--(90,30)
  (0,55) node {(b)} 
  ;

\end{tikzpicture}
};

\node at (8.4,0) {
%\begin{tikzpicture}[xscale=.035,yscale=.03031]
\begin{tikzpicture}[xscale=.03,yscale=.02598]

\draw
  (0,24)--(12,0)--(60,0)--(72,24)--(54,60)--(18,60)--(0,24)--(72,24) (12,0)--(24,24)--(36,0)--(48,24)--(60,0) (54,60)--(36,24)--(18,60)
  (0,55) node {(c)} 
  ;

\end{tikzpicture}
};

\node at (-.3,-2.5) {
%\begin{tikzpicture}[xscale=.035,yscale=.03031]
\begin{tikzpicture}[xscale=.03,yscale=.02598]

\draw
  (0,30)--(15,0)--(30,30)--(15,60)--(0,30)--(30,30) (30,20)--(40,0)--(60,0)--(70,20)--(50,60)--(30,20)--(70,20) (40,0)--(50,20)--(60,0)
  (0,55) node {(d)} 
  ;

\draw[dotted]
  (15,0)--(40,0) (15,60)--(50,60)
  ;

\end{tikzpicture}
};

\node at (3.37,-2.5) {
%\begin{tikzpicture}[xscale=.035,yscale=.03031]
\begin{tikzpicture}[xscale=.03,yscale=.02598]

\draw
  (0,30)--(15,0)--(55,0)--(65,20)--(45,60)--(15,60)--(0,30)--(30,30)--(15,60) (15,0)--(45,60) (65,20)--(25,20)--(35,0)--(45,20)--(55,0)
  (0,55) node {(e)} 
  ;

\end{tikzpicture}
};

\node at (7.1,-2.5) {
%\begin{tikzpicture}[xscale=.035,yscale=.03031]
\begin{tikzpicture}[xscale=.03,yscale=.02598]

\draw
  (0,30)--(15,0)--(65,0)--(75,20)--(55,60)--(15,60)--(0,30)--(30,30)--(15,0) (15,60)--(45,0)--(55,20)--(65,0) (75,20)--(35,20)--(55,60)
  (0,55) node {(f)} 
  ;

\end{tikzpicture}
};

\end{tikzpicture}
\end{center}
\caption{Proof of \eqref{eq:uphex}, Case 2}
\label{fig:case2}
\end{figure}
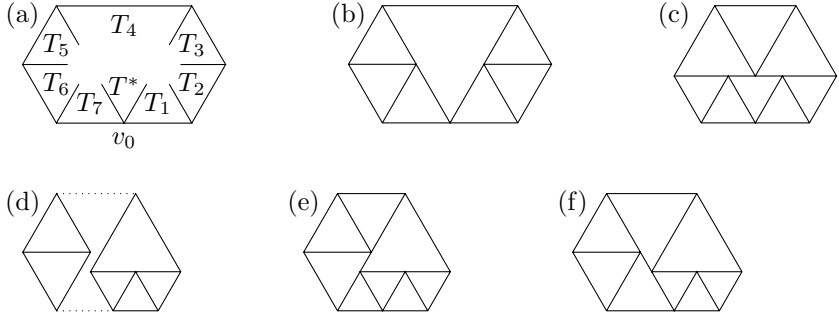
If $n=7$ then necessarily $T^\ast=T_4$, the hexagon $P$ splits into $T_4$ and two trapezoids tiled by $T_1,T_2,T_3$ and $T_5,T_6,T_7$, respectively, Lemma~\ref{lem:trapezoid} shows that $\mathcal{T}$ is similar to Figure~\ref{fig:case2}(b), and we obtain our claim $s(\mathcal{T})\le 2$.

Now we consider $n=8$. If $v_0$ was a vertex of $T_4$ then one of the two trapezoids besides $T_4$ would be tiled by four triangles, contrary to \eqref{eq:dom4t}. Thus $v_0$ is a vertex of $T_8 \in \mathcal{T}$. Consequently, the union of the lower sides of $T_3$ and $T_5$ is the union of the upper sides of $T_2$, $T_6$ and $T_8$.

If the lower sides of $T_3$ and $T_5$ together form a segment, then $\SIZE(T_3)=\SIZE(T_5)$ and $\SIZE(T_2)=\SIZE(T_6)=\SIZE(T_8)$, the tiling $\mathcal{T}$ is similar to Figure~\ref{fig:case2}(c), and we obtain $s(\mathcal{T})=2<3$.

If the lower sides of $T_3$ and $T_5$ do not form a segment, then w.l.o.g. the lower side of $T_3$ is the union of the upper sides of $T_2$ and $T_8$ and the lower side of $T_5$ coincides with the upper one of $T_6$, see Figure~\ref{fig:case2}(d). Then $\SIZE(T_1)=\SIZE(T_2)=\SIZE(T_8)$, $\SIZE(T_3)=2\SIZE(T_1)$ and $\SIZE(T_5)=\SIZE(T_6)=\frac{3}{2}\SIZE(T_1)$. Now it follows easily that either $\SIZE(T_4)=\frac{3}{2}\SIZE(T_1)$ (and in turn $\SIZE(T_7)=\SIZE(T_1)$, see Figure~\ref{fig:case2}(e)) or $\SIZE(T_4)=2\SIZE(T_1)$ (and in turn $\SIZE(T_7)=\frac{3}{2}\SIZE(T_1)$, see Figure~\ref{fig:case2}(f)). In both the situations we obtain our claim $s(\mathcal{T}) \le 3$.

\emph{Case 3: $v_\pi=1$ and $n \ge 9$. } Lemma~\ref{lem:necessary} gives at least $6+v_\pi-3=4$ pairs of triangles in $\mathcal{T}$ that have a common side. This yields $s(\mathcal{T}) \le |\mathcal{T}|-4=n-4$.

\emph{Case 4: $v_\pi \ge 2$ and $n \ge 7$. } Lemma~\ref{lem:necessary} shows that there are at least $6+v_\pi-3 \ge 5$ pairs of tiles in $\mathcal{T}$ having a side in common. This gives 
$$
s(\mathcal{T}) \le |\mathcal{T}|-5=n-5\le\left\{
\begin{array}{ll}
n-5, & n=7,8,\\
n-4, & n=9,10,\ldots
\end{array}
\right.
$$
\end{proof}

%%%%%%%%%%%%%%%%%%%%%%%%%%%%%%%%%%%%%%%%%%%%%%%%%%

\section{Proof of Theorem~\ref{thm:perfect}}

%%%

\subsection{Negative results on domains}

For 
\begin{equation}\label{eq:nDTRIp}
2,3,\ldots,14 \notin \DOM\left(\STRIp\right)
\end{equation}
we refer to \cite[Section 3.2]{drapal2010}.
It remains to show that
\begin{eqnarray}
1,2,\ldots,12 & \notin & \DOM\left(\STRAp\right), \label{eq:nDTRAp}\\
1;\; 3,4,\ldots,12 & \notin & \DOM\left(\SPARp\right), \label{eq:nDPARp}\\
1,2,\ldots,11 & \notin & \DOM\left(\SPENTp\right), \label{eq:nDPENTp}\\
1,2,\ldots,10 & \notin & \DOM\left(\SHEXp\right). \label{eq:nDHEXp}
\end{eqnarray}
These claims are prepared by a lemma.

\begin{lemma}\label{lem:perf_dom_neg}
\renewcommand{\labelenumi}{\bf (\alph{enumi})}
\begin{enumerate}
\item
Let $\ast$ stand for $\rm tri$, $\rm trap$, $\rm par$, $\rm pent$ or $\rm hex$. If $n \in \DOM\left(\SASTp\right)$ then $2\SAST(n) \ge  n$.

\item
If $n \in \DOM\left(\STRAp\right)$ then $n-2 \in \DOM\left(\STRAp\right) \cup \DOM\left(\SPENTp\right) \cup \DOM\left(\SHEXp\right)$.

\item
If $n \in \DOM\left(\SPARp\right)$ then\\ $n-2 \in \{0\} \cup \DOM\left(\STRAp\right) \cup \DOM\left(\SPARp\right) \cup \DOM\left(\SPENTp\right) \cup \DOM\left(\SHEXp\right)$.

\item
If $n \in \DOM\left(\SPENTp\right)$ then $n-1 \in \DOM\left(\STRAp\right) \cup \DOM\left(\SPENTp\right) \cup \DOM\left(\SHEXp\right)$.
\end{enumerate}
\end{lemma}

\begin{proof} 
(a) If $n \in \DOM\left(\SASTp\right)$ then there exists a corresponding  t-perfect tiling $\mathcal{T}$ of cardinality $n$. By t-perfectness, there are no three triangles of the same size in $\mathcal{T}$. Consequently, $n \le 2\SASTp(n)$. Combining this with the trivial estimate $\SASTp(n) \le \SAST(n)$ we obtain the claim $n \le 2\SAST(n)$.

(b) Let $n \in \DOM\left(\STRAp\right)$ be represented by a t-perfect tiling of some trapezoid $P$. Cutting off the two exposed triangles of the tiling from $P$ results in a smaller polygon $P'$ with a t-perfect tiling of cardinality $n-2$. The polygon $P'$ is neither a parallelogram nor a triangle, since in the latter case the two exposed triangles would have been congruent under a translation. So $P'$ is a trapezoid, a pentagon or a hexagon. Thus $n-2 \in \DOM\left(\STRAp\right) \cup \DOM\left(\SPENTp\right) \cup \DOM\left(\SHEXp\right)$.

(c) Now let $n \in \DOM\left(\SPARp\right)$ be represented by a t-perfect tiling of some parallelogram $P$. Cutting off the two exposed triangles from $P$, we get a polygon $P'$ that may be empty (if $n=2$) or a trapezoid, a parallelogram, a pentagon or a hexagon and has a t-perfect tiling of cardinality $n-2$. This yields (c).

(d) Here we start with a t-perfect tiling of cardinality $n$ of some pentagon $P$. We cut off the one exposed triangle and get a t-perfect tiling of cardinality $n-1$ of some trapezoid, pentagon or hexagon $P'$. 
\end{proof}

\begin{corollary}\label{cor:perf_neg}
\begin{eqnarray*}
1,2,\ldots,10 & \notin & \DOM\left(\STRAp\right),\\
1;\;3,4,\ldots,10 & \notin & \DOM\left(\SPARp\right),\\
1,2,\ldots,9 & \notin & \DOM\left(\SPENTp\right),\\
1,2,\ldots,8 & \notin & \DOM\left(\SHEXp\right).
\end{eqnarray*}
\end{corollary}

\begin{proof}
The trivial inclusion $\DOM\left(\SASTp\right)\subseteq\DOM(\SAST)$, Theorem~\ref{thm:general} and Lemma~\ref{lem:perf_dom_neg}(a) give
\begin{eqnarray*}
1,2,3,4,5,6,7 & \notin & \DOM\left(\STRAp\right),\\
1;\; 3,4,5,6,7 & \notin & \DOM\left(\SPARp\right),\\
1,2,3;\; 5 & \notin & \DOM\left(\SPENTp\right),\\
1,2,3,4,5,6,7,8 & \notin & \DOM\left(\SHEXp\right).
\end{eqnarray*}
Now fourfold application of Lemma~\ref{lem:perf_dom_neg}(d) yields
\[
4;\; 6,7,8 \notin \DOM\left(\SPENTp\right).
\]
Then Lemma~\ref{lem:perf_dom_neg}(b) and (c) show that
\[
8,9,10 \notin \DOM\left(\STRAp\right) \quad\mbox{ and }\quad
8,9,10 \notin \DOM\left(\SPARp\right).
\]
Finally, by Lemma~\ref{lem:perf_dom_neg}(d),
\[
9 \notin \DOM\left(\SPENTp\right).
\]
\end{proof}

\begin{lemma}\label{lem:perf_hex_dom_neg}
\renewcommand{\labelenumi}{\bf (\alph{enumi})}
$9,10 \notin \DOM\left(\SHEXp\right)$.
\end{lemma}

\begin{proof}
To obtain a contradiction, suppose that there is a t-perfect tiling $\mathcal{T}=\{T_1,\ldots,T_n\}$ of some convex hexagon $P$ with $n \in \{9,10\}$. Let $v_\pi$ be the number of $\pi$-vertices of $\mathcal{T}$.

\emph{Case 1: $v_\pi=0$. } As in Case 1 of the proof of \eqref{eq:uphex} in Subsection~\ref{subsec:upgen}, we see that $n=6$, a contradiction.

\emph{Case 2: $v_\pi \ge 3$. } Lemma~\ref{lem:necessary} says that $\mathcal{T}$ contains at least $6-v_\pi+3 \ge 6$ pairs of triangles that share a side. Since $n < 12$, there must be three triangles of the same size, which contradicts the t-perfectness of $\mathcal{T}$.

\emph{Case 3: $v_\pi=2$. } Now Lemma~\ref{lem:necessary} gives $6+v_\pi-3=5$ pairs of triangles with a common side. Since $\mathcal{T}$ is t-perfect, this shows that $n=10$ and the five pairs constitute five rhombi $R_1,\ldots,R_5$ of mutually different side lengths that tile $P$. 

For a vector $w \in \mathbb{R}^2 \setminus \{0\}$, we define the functional $\varphi_w(Q)$ of a polygon $Q$ by
\[
\varphi_w(Q)=\sum_{\genfrac{}{}{0pt}{2}{\mbox{\scriptsize $s$ is a side of $Q$ with}}{\mbox{\scriptsize outer normal vector $w$}}} \LENGTH(s)-
\sum_{\genfrac{}{}{0pt}{2}{\mbox{\scriptsize $s$ is a side of $Q$ with}}{\mbox{\scriptsize outer normal vector $-w$}}} \LENGTH(s).
\]
Clearly $\varphi_w(R_i)=0$ and, since $\varphi_w$ is additive under tiling, $\varphi_w(P)=\varphi_w(R_1)+\ldots+\varphi_w(R_5)=0$ for every $w$. That is, opposite sides of $P$ have the same length. By the assumption $v_\pi=2$, there are two opposite sides $s_1$, $s_2$ of $P$ that do not contain $\pi$-vertices of $\mathcal{T}$. So $s_1$ and $s_2$ are sides of two triangles $T_1,T_2 \in \mathcal{T}$ whose union is one of the rhombi, say $R_1$, because $\LENGTH(s_1)=\LENGTH(s_2)$. But then $R_1,\ldots,R_5$ cannot tile $P$, a contradiction.

\emph{Case 4: $v_\pi=1$. } Now we are in the situation of the left-hand side of Figure~\ref{fig:case4}, where $T_9$ (and possibly $T_{10}$) are not displayed.
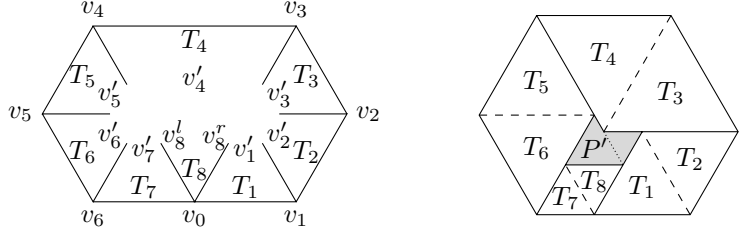
\begin{figure}
\begin{center}
\begin{tikzpicture}

\node at (0,-.22) {
%\begin{tikzpicture}[xscale=.05,yscale=.0433]
\begin{tikzpicture}[xscale=.045,yscale=.03897]

\draw
  (0,30)--(15,0)--(25,20) (15,0)--(45,0)--(35,20) (55,20)--(45,0)--(75,0)--(65,20) (75,0)--(90,30)--(70,30) (90,30)--(75,60)--(65,40) (75,60)--(15,60)--(25,40) (15,60)--(0,30)--(20,30) 
  (45,0) node[below] {$v_0$}
  (75,0) node[below] {$v_1$}
  (90,30) node[right] {$v_2$}
  (75,60) node[above] {$v_3$}
  (15,60) node[above] {$v_4$}
  (0,30) node[left] {$v_5$}
  (15,0) node[below] {$v_6$}
  (60,18) node {$v_1'$}
  (70,23) node {$v_2'$}
  (70,37) node {$v_3'$}
  (45,42) node {$v_4'$}
  (20,37) node {$v_5'$}
  (20,23) node {$v_6'$}
  (30,18) node {$v_7'$}
  (39,23) node {$v_8^l$}
  (51,22) node {$v_8^r$}
  (60,5) node {$T_1$}
  (78,17) node {$T_2$}
  (78,43) node {$T_3$}
  (45,55) node {$T_4$}
  (12,43) node {$T_5$}
  (12,17) node {$T_6$}
  (30,5) node {$T_7$}
  (45,11) node {$T_8$}
  ;

\end{tikzpicture}
};

\node at (5.5,-.25) {
\begin{tikzpicture}[xscale=.255,yscale=.22083]

\fill[black!15]
  (4.5,3)--(7.5,3)--(8.5,5)--(6.5,5)--(6,6)--cycle
  ;

\draw
  (3,12)--(0,6)--(3,0)--(11,0)--(13.5,5)--(10,12)--(3,12)--(6.5,5)--(13.5,5) (3,0)--(6,6) (6,0)--(8.5,5) (4.5,3)--(7.5,3)   
  (8.5,1.66) node {$T_1$} 
  (11,3.33) node {$T_2$} 
  (10,7.33) node {$T_3$} 
  (6.5,9.66) node {$T_4$} 
  (3,8) node {$T_5$} 
  (3,4) node {$T_6$} 
  (4.5,1) node {$T_7$} 
  (6,2) node {$T_8$} 
  (6,4) node {$P'$} 
  ;

\draw[dashed]
  (11,0)--(8.5,5) (10,12)--(6.5,5) (0,6)--(6,6) (6,0)--(4.5,3)
  ;

\draw[densely dotted]
  (6.5,5)--(7.5,3)
  ;

\end{tikzpicture}
};

\end{tikzpicture}
\end{center}
\caption{Proof of Lemma~\ref{lem:perf_hex_dom_neg}, Case 4}
\label{fig:case4}
\end{figure}
Clearly $T_4 \ne T_8$, because otherwise the trapezoids $v_0v_1v_2v_3$ and $v_0v_4v_5v_6$ would be tiled by a total number of eight or nine triangles, a contradiction to Corollary~\ref{cor:perf_neg}. The $2\pi$-vertices of the triangles $T_1,\ldots,T_8$ are denoted by $v_1',\ldots,v_7',v_8^l,v_8^r$. We will discuss coincidences between these vertices. 

First note that two of these nine points cannot agree if their lower indices differ by more than one modulo $8$: Indeed, if a $2\pi$-vertex of $T_i$ coincides with a $2\pi$-vertex of $T_j$ such that $T_i$ and $T_j$ are no neighbours in the cyclic order of $T_1,\ldots,T_8$, then there is a polygonal arc $a$ along sides of $T_i$ and $T_j$ that connects a vertex of $T_i$ with a vertex of $T_j$ and passes through their common $2\pi$-vertex such that $a$ dissects $P$ into two polygons $P'$ and $P''$ both being tiled by at least three (and in turn at most seven) triangles of $\mathcal{T}$ and such that one of $P'$ and $P''$ is convex. (E.g., if $v_1'=v_3'$ we can pick $a=v_0v_1'v_3$, if $v_2'=v_8^r$ we can pick $a=v_2v_2'v_8^lv_0$.) Thus some convex polygon admits a t-perfect tiling of a cardinality between $3$ and $7$. This is impossible by \eqref{eq:nDTRIp} and Corollary~\ref{cor:perf_neg}. Consequently, \emph{the only possible coincidences among $v_1',\ldots,v_7',v_8^l,v_8^r$ are
\[
v_1'=v_2',\;v_2'=v_3',\;v_3'=v_4',\;v_4'=v_5',\;v_5'=v_6',\;v_6'=v_7',\;v_7'=v_8^l,\;v_8^r=v_1'.
\]
When counted cyclically, no two consecutive of these identities are satisfied,} because they would give rise to three tiles of the same size, which contradicts t-perfectness.

The total size of all inner angles of $T_1,\ldots,T_9(,T_{10})$ is $9\pi$ or $10\pi$. The size of those who are placed on the boundary of $P$ is $5\pi$. Thus the sizes of inner angles in $2\pi$-vertices of $\mathcal{T}$ sum up to $4\pi$ or $5\pi$. Since every $2\pi$-vertex requires inner angles of total size $\pi$ or $2\pi$, $\mathcal{T}$ has at most five $2\pi$-vertices. All of $v_1',\ldots,v_7',v_8^l,v_8^r$ are $2\pi$-vertices. So there are at least four identities between them. Taking the above observation into account, we obtain
\begin{equation}
v_1'=v_2',\quad v_3'=v_4',\quad v_5'=v_6',\quad v_7'=v_8^l
\end{equation}
(the alternative situation $v_2'=v_3',v_4'=v_5',v_6'=v_7',v_8^r=v_1'$ would be equivalent).

So $T_1 \cup T_2$, $T_3 \cup T_4$, $T_5 \cup T_6$ and $T_7 \cup T_8$ are rhombi and the sizes of $T_1,T_3,T_5,T_7$ are mutually different. Since $\SIZE(T_2)+\SIZE(T_3)=\SIZE(T_5)+\SIZE(T_6)=2\SIZE(T_5)$, we have either $\SIZE(T_2) < \SIZE(T_5) < \SIZE(T_3)$ or $\SIZE(T_2) > \SIZE(T_5) > \SIZE(T_3)$. As the latter situation would cause an overlap of $T_2$ and $T_5$, the former inequality applies and we are in the situation of the right-hand side of Figure~\ref{fig:case4}. The remainder $P'$ of $P$ after cutting off $T_1,\ldots,T_8$ must be a pentagon with inner angles of sizes $\frac{2\pi}{3}$, $\frac{\pi}{3}$, $\frac{4\pi}{3}$, $\frac{\pi}{3}$ and $\frac{\pi}{3}$ in consecutive order.
This excludes $n=9$, because $P'$ cannot be tiled by one single triangle $T_9$. If $n=10$, the only possibility of tiling $P'$ by $T_9$ and $T_{10}$ is dotted in the figure. But then the tiling of $P$ is not t-perfect. This contradiction completes the proof.
\end{proof}

\begin{proof}[Proof of \eqref{eq:nDTRAp}, \eqref{eq:nDPARp}, \eqref{eq:nDPENTp} and \eqref{eq:nDHEXp}.]
By Corollary~\ref{cor:perf_neg} and Lemma~\ref{lem:perf_hex_dom_neg}, it remains to prove that
\begin{eqnarray*}
11,12 \notin \DOM\left(\STRAp\right), \quad
11,12 \notin \DOM\left(\SPARp\right), \quad
10,11 \notin \DOM\left(\SPENTp\right).
\end{eqnarray*}
The third claim follows from Corollary~\ref{cor:perf_neg} and Lemma~\ref{lem:perf_hex_dom_neg} by twofold application of Lemma~\ref{lem:perf_dom_neg}(d). Then the remainder is a consequence of Lemma~\ref{lem:perf_dom_neg}(b) and (c).
\end{proof}

%%%

\subsection{New spiral pentagons and corresponding t-perfect tilings}

We define a sequence of integers $q(n)$, $n=8,9,\ldots$, by
\begin{equation}\label{eq:t-Pad}
q(8)=8,\; q(9)=11,\; q(10)=9 \quad\mbox{and}\quad q(n)=q(n-3)+q(n-2) \mbox{ for } n \ge 11.
\end{equation}

\begin{lemma}\label{lem:t-spiral}
\renewcommand{\labelenumi}{\bf (\alph{enumi})}
\begin{enumerate}
\item
The above numbers satisfy 
\begin{itemize}
\item
$q(n)>q(n-1)$ and $q(n)>12$ for $n \ge 11$,
\item
$\frac{1}{3}(q(n-1)-q(n-2)) \notin \{2,3,5,7,8,9,11,12\} \cup \{q(m): m \ge 8\}$ for $n=12,13,14,15;\; 17,18,\ldots$
\end{itemize}

\item
For every $n \in \{12,13,\ldots\}$, there is a convex pentagon $Q_n$ with sides of lengths $q(n-4)$, $q(n-3)$, $q(n-2)$, $q(n-1)$ and $q(n)$ (in consecutive order) with an inner angle of size $\frac{\pi}{3}$ between the sides of lengths $q(n-1)$ and $q(n)$. That pentagon admits a t-perfect tiling $\mathcal{T}_n$ by $n$ equilateral triangles $T_i$, $i=1,\ldots,n$, such that
\begin{itemize}
\item
$\SIZE(T_1)=\SIZE(T_2)=2$, $\SIZE(T_3)=3$, $\SIZE(T_4)=\SIZE(T_5)=5$, $\SIZE(T_6)=\SIZE(T_7)=7$, $\SIZE(T_8)=\SIZE(T_9)=8$, $\SIZE(T_{10})=9$, $\SIZE(T_{11})=11$, $\SIZE(T_{12})=12$ and $\SIZE(T_i)=q(i-1)$ for $i=13,\ldots,n$,
\item
if one side length of $Q_n$ is a side length of some triangle $T_i \in \mathcal{T}_n$ and different from $8$ (the latter is always the case if $n \ge 13$), then that side of $Q_n$ is a side of $T_i$ and there is only one triangle of that size in $\mathcal{T}_n$.
\end{itemize}
\end{enumerate}
\end{lemma}

\begin{proof}
The first part of (a) follows from \eqref{eq:t-Pad} by induction. The second part of (a) is shown for $n \le 24$ by explicit computation. For $n \ge 25$, it is proved by inductive verification of 
\[
q(n-10) < \frac{1}{3}(q(n-1)-q(n-2))=\frac{q(n-6)}{3} < q(n-9)
\] 
as follows: First note that 
\[
q(n-1)-q(n-2)=q(n-4)+q(n-3)-q(n-5)-q(n-4)=q(n-3)-q(n-5)=q(n-6).
\]
The base cases for $n=25,26,27$, obtained by computation, are
$q(15)=48 < \frac{q(19)}{3}=\frac{154}{3} < q(16)=67 < \frac{q(20)}{3}=\frac{202}{3} < q(17)=87 < \frac{q(21)}{3}=\frac{269}{3} < q(18)=115$. Finally, the step case for $n \ge 28$ is
\[
\frac{q(n-6)}{3}=\frac{q(n-9)}{3}+\frac{q(n-8)}{3}
\left\{
\begin{array}{l}
>q(n-13)+q(n-12)=q(n-10),\\[.5ex] 
< q(n-12)+q(n-11)=q(n-9).
\end{array}
\right.
\]

Also claim (b) is shown by induction: The pentagon $Q_{12}$ and the tiling $\mathcal{T}_{12}$ is illustrated in Figure~\ref{fig:t-spiral} (see also tiling (c) in the appendix). 
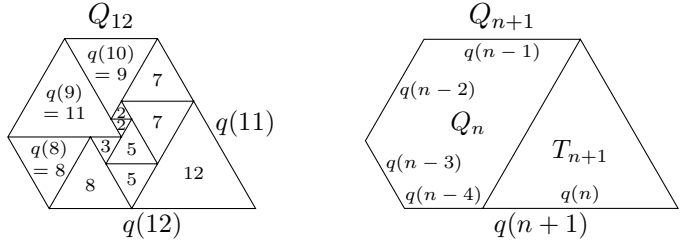
\begin{figure}
\begin{center}
\begin{tikzpicture}

\node at (-5,0) {
\begin{tikzpicture}[xscale=.0686,yscale=.0595]
%\begin{tikzpicture}[xscale=.075,yscale=.065]
\draw 
  (8,0)--(48,0)--(29,38)--(11,38)--(0,16)--(8,0)--(16,16)--(24,0)--(36,24)--(22,24)--(29,10)--(19,10)--(24,20)--(20,20)--(29,38)
  (0,16)--(22,16)--(11,38)
  (20,43) node {\large $Q_{12}$}
  (28,-3.8) node {$q(12)$}
  (46,19) node {$q(11)$}
  (29,28.67) node {\scriptsize $7$}
  (29,19.33) node {\scriptsize $7$}
  (36,8) node {\scriptsize $12$}
  (16,5.33) node {\scriptsize $8$}
%  (19,34.67) node {\scriptsize $8$}
  (8,10.67) node {\scriptsize $\begin{array}{c} q(8)\\ =8 \end{array}$}
  (11,23.33) node {\scriptsize $\begin{array}{c} q(9)\\ =11 \end{array}$}
  (20,32) node {\scriptsize $\begin{array}{c} q(10)\\ =9 \end{array}$}
  (22,18.6) node {\tiny $2$}
  (22,21.4) node {\tiny $2$}
  (24,13.33) node {\scriptsize $5$}
  (24,6.67) node {\scriptsize $5$}
  (19,14) node {\scriptsize $3$}
  ;
\end{tikzpicture}};

\node at (0,0) {
\begin{tikzpicture}[xscale=.1299,yscale=.15]
\draw 
  (22,15)--(6,15)--(0,6)--(4,0)--(32,0)--(22,15)--(12,0)
  (14,17) node {\large $Q_{n+1}$}
  (10.5,7.5) node {$Q_n$}
  (8,1.2) node {\scriptsize $q(n-4)$}
  (6,4) node {\scriptsize $q(n-3)$}
  (7.4,10.5) node {\scriptsize $q(n-2)$}
  (14,13.8) node {\scriptsize $q(n-1)$}
  (22,1.2) node {\scriptsize $q(n)$}
  (18,-1.5) node {$q(n+1)$}
  (22,5) node {$T_{n+1}$}
  ;
\end{tikzpicture}};

\end{tikzpicture}
\end{center}
\caption{Pentagons $Q_{12}$ and $Q_{n+1}$}
\label{fig:t-spiral}
\end{figure}
Then $Q_{n+1}$ and $\mathcal{T}_{n+1}$ are obtained from $Q_n$ and $\mathcal{T}_n$ by adding a triangle $T_{n+1}$ of size $q(n)$ at the side of size $q(n)$ of $Q_n$, cf.\ Figure~\ref{fig:t-spiral}. 
\end{proof}

%%%

\subsection{Positive results on domains and lower estimates}

We start with a direct consequence of Lemma~\ref{lem:t-spiral}.

\begin{corollary}\label{cor:t-lower}
\renewcommand{\labelenumi}{\bf (\alph{enumi})}
\begin{enumerate}
\item 
$\{15,16,\ldots\}\subseteq\DOM\left(\STRIp\right)$ and $\STRIp(n) \ge n-6$ for all $n=15,16,\ldots$

\item
$\{13,14,\ldots\}\subseteq\DOM\left(\STRAp\right)$ and $\STRAp(n)\ge n-5$ for all $n=13,14,\ldots$

\item
$\{13,14,\ldots\}\subseteq\DOM\left(\SPARp\right)$ and $\SPARp(n)\ge n-5$ for all $n=13,14,\ldots$

\item
$\{12,13,\ldots\}\subseteq\DOM\left(\SPENTp\right)$ and
$\SPENTp(n)\ge n-4$ for all $n=12,13,\ldots$

\item
$\{17,18,\ldots\} \setminus \{21\} \subseteq\DOM\left(\SHEXp\right)$ and $\SHEXp(n) \ge n-6$ for all $n=17,18,19,20;\; 22,23,\ldots$
\end{enumerate}
\end{corollary}

\begin{proof}
The pentagons $Q_n$ and tilings $\mathcal{T}_n$ from Lemma~\ref{lem:t-spiral}(b) prove (d). 

For (a), we add two triangles of sizes $q(n-4)$ and $q(n-2)$ at the respective sides of $Q_n$, cf.\ Figure~\ref{fig:derived polygons}(a). This gives a t-perfect tiling of a triangle by $n+2$ tiles of at least $n-4=(n+2)-6$ different sizes if $n \ge 13$. (For $n=12$, t-perfectness is violated, since the size $q(12-4)=8$ is not allowed, see the last claim of Lemma~\ref{lem:t-spiral}(b).)

For (b) and (c), we add a triangle of size $q(n-2)$ or $q(n-3)$, respectively, to $Q_n$ and $\mathcal{T}_n$ for $n \ge 12$, cf.\ Figure~\ref{fig:derived polygons}.

For (e), we add five triangles $T'_{n+1},\ldots,T'_{n+5}$ of sizes
\[
\SIZE(T'_{k})=\left\{
\begin{array}{ll}
\frac{1}{3}\big(q(n-1)-q(n-2)\big),& k=n+1, \\[1ex]
\frac{2}{3}q(n-2)+\frac{1}{3}q(n-1),& k=n+2,n+3,\\[1ex]
\frac{1}{3}q(n-2)+\frac{2}{3}q(n-1),& k=n+4,n+5,\\
\end{array}
\right.
\]
over the sides of lengths $q(n-2)$ and $q(n-1)$ of $Q_n$, see Figure~\ref{fig:thex}.
\begin{figure}
\begin{center}
\begin{tikzpicture}[xscale=.0577,yscale=.2]

\draw 
  (52,0)--(68,8)--(52,16)--(24,16)--(0,4)--(8,0)--(52,0)--(34,9) 
  (10,9)--(38,9)--(24,16)
  (52,16)--(36,8)--(68,8)
  (36,8.5)--(39.3,8.9)
  (12,3.5) node {$Q_n$}
  (45,8.9) node {\scriptsize $T'_{n+1}$}
  (24,11.33) node {$T'_{n+2}$}
  (38,13.67) node {$T'_{n+3}$}
  (52,10.67) node {$T'_{n+4}$}
  (52,5.33) node {$T'_{n+5}$}
  (22,8) node {\scriptsize $q(n-2)$}
  (32.5,4.5) node {\scriptsize $q(n-1)$}
  (30,1) node {\scriptsize $q(n)$}
  ;

\end{tikzpicture}
\end{center}
\caption{Proof of Corollary~\ref{cor:t-lower}(e)}
\label{fig:thex}
\end{figure}
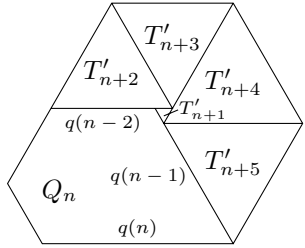
By Lemma~\ref{lem:t-spiral}(a), $\SIZE(T'_{n+1}) \ne \SIZE(T_i)$ for $i=1,\ldots,n$ if $n \in \{12,13,\ldots\} \setminus \{16\}$. Moreover, the monotonicity from Lemma~\ref{lem:t-spiral}(a) gives
\[
q(n-2) < \SIZE(T'_{n+2})=\SIZE(T'_{n+3}) < \SIZE(T'_{n+4})=\SIZE(T'_{n+5}) < q(n-1),
\] 
so that the sizes of the new triangles are different from those of $\mathcal{T}_n$. Thus $\mathcal{T}_n$ together with $T'_{n+1},\ldots,T'_{n+5}$ constitutes a t-perfect tiling by $n+5$ triangles of $(n-4)+3=(n+5)-6$ different sizes. This proves (e). 
\end{proof}

Now the remaining (non-trivial) claims from Theorem~\ref{thm:perfect} on domains and lower estimates are
\begin{align}
&11;\;14,15,16;\;21 \in \DOM\left(\SHEXp\right),\label{eq:dthex}\\
&\STRIp(n) \ge n-5, \quad n=15;\; 17,18,\ldots,26;\; 28,\label{eq:sttri}\\
&\STRAp(n) \ge n-4, \quad n= 14;\; 16,17,\ldots,25;\; 27,\label{eq:sttra}\\
&\SPARp(n)  \ge n-4, \quad n= 15;\; 18,19;\; 21,22,23;\; 26,\label{eq:stpar}\\
&\SHEXp(n) \ge \left\{
\begin{array}{ll}
n-4, & n=11;\; 14,15;\; 17,18,19;\; 22,\\
n-5, & n=16;\; 20,21;\; 23.
\end{array}
\right.
\label{eq:sthex}
\end{align}

\begin{proof}[Proof of \eqref{eq:dthex} and \eqref{eq:sthex}]
It is enough to show that the lower bounds from \eqref{eq:sthex} are attained by t-perfect tilings of the respective cardinalities. We refer to the tilings (c), (j), (k), (m), (n), (o), (q) and (t), (u), (v), (w) from the appendix and Table~\ref{tab:hexagons}. 
\end{proof}

\begin{proof}[Proof of \eqref{eq:sttri}, \eqref{eq:sttra} and \eqref{eq:stpar}]
These estimates are established by tilings constructed as follows: We start with a particular tiling from the appendix. Then we add successively triangles of prescribed sizes such that in each step the new triangle is placed over the side of the respective length of the previously tiled polygon. Table~\ref{tab:parameters} summarizes all parameters. 

Several of the above tilings showing that $\STRIp(n) \ge n-5$ can be found in the literature: for $n=15$ see \cite[Figure 1 and the text thereafter]{tutte1948}, \cite[Figure 4]{tuza1991}, \cite[Figure 8, second tiling]{drapal2010} and \cite[perfect\underline{ }dissection\underline{ }size15\underline{ }595\underline{ }r5\underline{ }c3]{hamalainen}, for $n=17$ see \cite[Figure 10, first tiling]{drapal2010} and \cite[perfect\underline{ }dissection\underline{ }size17\underline{ }3091\underline{ }r3\underline{ }c0]{hamalainen}, for $n=18$ see \cite[perfect\underline{ }dissection\underline{ }size18\underline{ }30413\underline{ }r6\underline{ }c3]{hamalainen}, for $n=19$ see \cite[perfect\underline{ }dissection\underline{ }size19\underline{ }130975\underline{ }r6\underline{ }c4]{hamalainen}. 
\end{proof}

\begin{table}
\caption{Tilings proving claims \eqref{eq:sttri}, \eqref{eq:sttra} and
\eqref{eq:stpar}: the table gives the respective initial tilings from the appendix and the sizes of new triangles in the order of their addition.\label{tab:parameters}}

\begin{tabular}{|c|p{50mm}||c|p{50mm}|}
\hline
\multicolumn{2}{|c||}{Claim \eqref{eq:sttri}} & \multicolumn{2}{|c|}{Claim \eqref{eq:sttra}}\\
\hline
$n$ && $n$ &\\
\hline
15 & (c); 12, 19, 20, 11 &
14 & (c); 12, 19, 20 \\
17 & (c); 12, 19, 28, 39, 20, 28 &
16 & (c); 12, 19, 28, 39, 20\\
18 & (j); 27, 44, 47, 24 &
17 & (j); 27, 44, 47\\
19 & (k); 20, 32, 33, 17 &
18 & (k); 20, 32, 33\\
20 & (j); 27, 44, 67, 91, 47, 67 &
19 & (j); 27, 44, 67, 91, 47\\
21 & (m); 64, 106, 111, 59 &
20 & (m); 64, 106, 111\\
22 & (n); 102, 157, 162, 84 &
21 & (n); 102, 157, 162\\
23 & (o); 138, 213, 220, 114 &
22 & (o); 138, 213, 220\\
24 & (n); 102, 157, 235, 319, 162, 235 &
23 & (n); 102, 157, 235, 319, 162\\
25 & (o); 138, 213, 319, 433, 220, 319 &
24 & (o); 138, 213, 319, 433, 220\\
26 & (q); 325, 533, 534, 283 &
25 & (q); 325, 533, 534\\
28 & (q); 325, 533, 784, 1067, 534, 784 &
27 & (q); 325, 533, 784, 1067, 534\\
\hline\hline
\multicolumn{4}{|c|}{Claim \eqref{eq:stpar}}\\
\hline
$n$ && $n$ &\\
\hline
15 & (c); 12, 19, 28, 20 &
22 & (n); 102, 157, 235, 162\\
18 & (j); 27, 44, 67, 47 &
23 & (o); 138, 213, 319, 220 \\
19 & (k); 20, 32, 48, 33 &
26 & (q); 325, 533, 784, 534 \\
21 & (m); 64, 106, 158, 111 &&\\
\hline
\end{tabular}
\end{table}

%%%

\subsection{Upper estimates}

The property of t-perfectness gives a new upper estimate for pentagons.

\begin{lemma}\label{lem:necessary_t-perfect}
$\SPENTp(n) \le n-4$ for all $n \in \DOM\left(\SPENTp\right)$.
\end{lemma}

\begin{proof}
Assume that Lemma~\ref{lem:necessary_t-perfect} fails. Let $n_0 \in \DOM\left(\SPENTp\right) \stackrel{\mbox{\eqref{eq:nDPENTp}}}{\subseteq}\{12,13,\ldots\}$ be minimal such that $\SPENTp(n_0) \ge n_0-3$, and let $\mathcal{T}=\{T_1,\ldots,T_{n_0}\}$ be a tiling of some convex pentagon $P$ realizing $s(\mathcal{T})=\SPENTp(n_0)$. Cutting off the exposed triangle, we obtain a t-perfect tiling $\mathcal{T}'$ of some convex polygon $P'$ with $|\mathcal{T}'|=n_0-1$ and 
\[
s(\mathcal{T}') \ge s(\mathcal{T})-1 = \SPENTp(n_0)-1 \ge (n_0-3)-1= |\mathcal{T}'|-3.
\]
By $|\mathcal{T}'|=n_0-1 \ge 11$, Theorem~\ref{thm:general} tells us that $s(\mathcal{T}')=|\mathcal{T}'|-3$ and $P'$ is a pentagon. This contradicts the minimality of $n_0$.
\end{proof}

The estimate $\STRIp(16)\le 10$ from Theorem~\ref{thm:perfect}(a) is shown in \cite[Section~3.2 and Figure~9]{drapal2010}. The remaining upper estimates in Theorem~\ref{thm:perfect} follow from Theorem~\ref{thm:general} and the trivial relations $\SASTp(n) \le \SAST(n)$ for $n \in \DOM\left(\SASTp\right)$.

%%%%%%%%%%%%%%%%%%%%%%%%%%%%%%%%%%%%%%%%%%%%%%%%%%

\section*{Appendix. Particular tilings of hexagons}

The following illustrations are scaled such that the smallest tiles appear congruent. Sizes of larger tiles are indicated in the figures. Sizes of all tiles are summarized in Table~\ref{tab:hexagons}.
\\[2ex]
\begin{tikzpicture}[xscale=.09,yscale=.0779]
\draw
  (0,3)--(3,9)--(8,9)--(10.5,4)--(8.5,0)--(1.5,0)--(0,3)--(6,3)--(3,9)
  (8,9)--(5.5, 4)--(10.5,4)
  (1.5,0)--(3,3)--(4.5,0)--(6.5,4)--(8.5,0)
  ;
%\node at (4.5,2) {\tiny $3$};  
\node at (6.5,1.33) {\scriptsize $4$};  
\node at (8.5,2.67) {\scriptsize $4$};  
\node at (5.5,7.33) {\small $5$};  
\node at (8,5.67) {\small $5$};  
\node at (3,5) {$6$};    
\end{tikzpicture}(a)\quad
%\hfill
\begin{tikzpicture}[xscale=.09,yscale=.0779]
\draw
  (0,4)--(4,12)--(11,12)--(14,6)--(11,0)--(2,0)--(0,4)--(8,4)--(4,12)
  (2,0)--(4,4)--(6,0)--(8.5,5)--(7.5,5)--(11,12)
  (14,6)--(8,6)--(11,0)
  ;
\node at (2,2.67) {\scriptsize $4$};  
\node at (4,1.33) {\scriptsize $4$};  
\node at (6,2.67) {\scriptsize $4$};  
\node at (8.5,1.67) {\small $5$};  
\node at (11,4) {$6$};  
\node at (11,8) {$6$};  
\node at (7.5,9.67) {$7$};  
\node at (4,6.67) {$8$};  
\end{tikzpicture}(b)\quad
%\hfill
\begin{tikzpicture}[xscale=-.045,yscale=.0389]
\draw
  (0,8)--(4,16)--(15,16)--(19.5,7)--(16,0)--(4,0)--(0,8)--(8,8)--(4,0)
  (19.5,7)--(10.5,7)--(11.5,5)--(6.5,5)--(9,0)--(12.5,7)--(16,0)
  (15,16)--(9.5,5)--(4,16)
  ;
\node at (9.5,12.67) {\scriptsize $11$};    
\node at (15,10) {\footnotesize $9$};    
\node at (4,10.67) {\scriptsize $8$};    
\node at (4,5.33) {\scriptsize $8$};    
\node at (16,4.67) {\tiny $7$};    
\node at (12.5,2.33) {\tiny $7$};    
\end{tikzpicture}(c)\quad
%\hfill
\begin{tikzpicture}[xscale=.09,yscale=.0779]
\draw
  (0,7)--(5,17)--(14,17)--(18.5,8)--(14.5,0)--(3.5,0)--(0,7)--(10,7)--(5,17)
  (3.5,0)--(7,7)--(10.5,0)--(12.5,4)--(8.5,4)--(10.5,8)--(14.5,0)
  (14,17)--(9.5,8)--(18.5,8)
  ;
%\node at (8.5,6) {\tiny $3$};  
\node at (10.5,5.33) {\scriptsize $4$};  
\node at (10.5,2.67) {\scriptsize $4$};  
\node at (12.5,1.33) {\scriptsize $4$};  
\node at (3.5,4.67) {$7$};  
\node at (7,2.33) {$7$};  
\node at (14.5,5.33) {$8$};  
\node at (14,11) {$9$};  
\node at (9.5,14) {$9$};
\node at (5,10.33) {$10$};     
\end{tikzpicture}(d)\quad
%\hfill
\begin{tikzpicture}[xscale=.09,yscale=.0779]
\draw
  (0,10)--(6.5,23)--(18.5,23)--(24.5,11)--(19,0)--(5,0)--(0,10)--(13,10)--(6.5,23)
  (5,0)--(10,10)--(11.5,7)--(13.5,11)--(19,0)
  (12,0)--(8.5,7)--(15.5,7)--(12,0)
  (18.5,23)--(12.5,11)--(24.5,11)
  ;
\node at (13.5,8.33) {\scriptsize $4$};  
\node at (8.5,2.33) {$7$};  
\node at (12,4.67) {$7$};  
\node at (15.5,2.33) {$7$};  
\node at (5,6.67) {$10$};  
\node at (19,7.33) {$11$};  
\node at (18.5,15) {$12$};  
\node at (12.5,19) {$12$};  
\node at (6.5,14.33) {$13$};  
\end{tikzpicture}(e)\quad
%\hfill
\\[2ex]
\begin{tikzpicture}[xscale=.09,yscale=.0779]
\draw
  (0,7)--(7,21)--(19,21)--(24.5,10)--(19.5,0)--(3.5,0)--(0,7)--(14,7)--(7,21)
  (3.5,0)--(7,7)--(10.5,0)--(15,9)--(13,9)--(13.5,10)--(14,9)--(14.5,10)--(19.5,0)
  (19,21)--(13.5,10)--(24.5,10)
  ;
\node at (3.5,4.67) {$7$};  
\node at (7,2.33) {$7$};  
\node at (10.5,4.67) {$7$};  
\node at (15,3) {$9$};  
\node at (19.5,6.67) {$10$};  
\node at (19,13.67) {$11$};  
\node at (13,17) {$12$};  
\node at (7,11.67) {$14$};  
\end{tikzpicture}(f)\quad
%\hfill
\begin{tikzpicture}[xscale=-.045,yscale=.0389]
\draw
  (0,14)--(7,28)--(26,28)--(33.5,13)--(27,0)--(7,0)--(0,14)--(14,14)--(7,0)
  (27,0)--(20.5,13)--(19.5,11)--(21.5,11)--(16,0)--(11.5,9)--(20.5,9)--(18.5,13)--(33.5,13)
  (7,28)--(16.5,9)--(26,28)
  ;
\node at (11.5,3) {\scriptsize $9$};  
\node at (16,6) {\scriptsize $9$};  
\node at (21.5,4) {\tiny $11$};  
\node at (27,8.67) {\footnotesize $13$};  
\node at (7,9.33) {\footnotesize $14$};  
\node at (7,18.67) {\footnotesize $14$};  
\node at (26,18) {\small $15$};  
\node at (16.5,21.67) {$19$};  
\end{tikzpicture}(g)\quad
%\hfill
\begin{tikzpicture}[xscale=.03,yscale=.0259]
\draw
(0,16)--(10.5,37)--(36.5,37)--(46.5,17)--(38,0)--(8,0)--(0,16)--(21,16)--(18.5,11)--(29.5,11)--(26.5,17)--(46.5,17)
  (8,0)--(16,16)--(24,0)--(31,14)--(28,14)--(29.5,17)--(38,0)
  (10.5,37)--(23.5,11)--(36.5,37)
  ;
%\node at (31,4.67) {\tiny $14$};  
\node at (8,10.67) {\tiny $16$};  
\node at (16,5.33) {\tiny $16$};  
\node at (38,11.33) {\scriptsize $17$};  
\node at (36.5,23.67) {\small $20$};  
\node at (10.5,23) {\small $21$};  
\node at (23.5,28.33) {$26$};  
\end{tikzpicture}(h)\quad
%\hfill
\begin{tikzpicture}[xscale=.045,yscale=.0389]
\draw
  (0,17)--(11.5,40)--(32.5,40)--(43,19)--(33.5,0)--(8.5,0)--(0,17)--(23,17)--(11.5,40)
  (8.5,0)--(17,17)--(20,11)--(24,19)--(33.5,0)
  (22.5,0)--(15.5,14)--(18.5,14)--(17,11)--(28,11)--(22.5,0)
  (32.5,40)--(22,19)--(43,19)
  ;
%\node at (20,15) {\tiny $6$};  
\node at (24,13.67) {\tiny $8$};  
\node at (28,3.67) {\tiny $11$};  
\node at (22.5,7.33) {\tiny $11$};  
\node at (15.5,4.67) {\small $14$};  
\node at (8.5,11.33) {$17$};  
\node at (33.5,12.67) {$19$};  
\node at (32.5,26) {$21$};  
\node at (22,33) {$21$};  
\node at (11.5,24.67) {$23$};  
\end{tikzpicture}(i)\quad
%\hfill
\\[2ex]
\begin{tikzpicture}[xscale=.09,yscale=.0779]
\draw (0,17)--(11.5,40)--(35.5,40)--(45.5,20)--(35.5,0)--(8.5,0)--(0,17)--(23,17)--(20,11)--(17,17)--(8.5,0)
  (19.5,0)--(14,11)--(25,11)--(22.5,16)--(27.5,16)--(19.5,0)
  (11.5,40)--(23.5,16)--(35.5,40)
  (45.5,20)--(25.5,20)--(35.5,0)
  ;
\node at (25.5,17.33) {\scriptsize $4$};  
\node at (25,14.33) {\small $5$};  
\node at (22.5,12.67) {\small $5$};  
\node at (20,15) {$6$};  
\node at (17,13) {$6$};  
\node at (14,3.67) {$11$};  
\node at (19.5,7.33) {$11$};  
\node at (27.5,5.33) {$16$};  
\node at (8.5,10.67) {$17$};  
\node at (35.5,13.33) {$20$};  
\node at (35.5,26.67) {$20$};  
\node at (11.5,24.33) {$23$};  
\node at (23.5,32) {$24$};   
\end{tikzpicture}(j)\quad
%\hfill
\begin{tikzpicture}[xscale=.09,yscale=.0779]
\draw (0,12)--(8,28)--(25,28)--(32.5,13)--(26,0)--(6,0)--(0,12)--(16,12)--(14,8)--(22,8)--(18,0)--(12,12)--(6,0)
  (26,0)--(19.5,13)--(17,8)--(15.5,11)--(18.5,11)--(17.5,13)--(32.5,13)
  (8,28)--(16.5,11)--(25,28)  
  ;
\node at (14,10.67) {\scriptsize $4$};  
\node at (19.5,9.67) {\small $5$};  
\node at (18,5.33) {$8$};  
\node at (22,2.67) {$8$};  
\node at (6,8) {$12$};  
\node at (12,4) {$12$};  
\node at (26,8.67) {$13$};  
\node at (25,18) {$15$};  
\node at (8,17.33) {$16$};  
\node at (16.5,22.33) {$17$};    
\end{tikzpicture}(k)\quad
%\hfill
\begin{tikzpicture}[xscale=.03,yscale=.0259]
\draw
  (0,30)--(19,68)--(62,68)--(79,34)--(62,0)--(15,0)--(0,30)--(38,30)--(34,22)--(30,30)--(15,0)
  (37,0)--(26,22)--(37,22)--(31.5,11)--(42.5,11)--(35.5,25)--(49.5,25)--(37,0)
  (19,68)--(40.5,25)--(62,68)
  (79,34)--(45,34)--(62,0)
  ;
%\node at (42.5,20.33) {\tiny $14$};  
\node at (26,7.33) {$22$};  
\node at (49.5,8.33) {$25$};  
\node at (15,20) {$30$};  
\node at (62,22.66) {$34$};  
\node at (62,45.33) {$34$};  
\node at (19,42.67) {$38$};  
\node at (40,54) {$43$};  
\end{tikzpicture}(l)\quad
%\hfill
\\[2ex]
\begin{tikzpicture}[xscale=.03,yscale=.0259]
\draw
(0,42)--(26,94)--(85,94)--(108.5,47)--(85,0)--(21,0)--(0,42)--(52,42)--(47,32)--(50,32)--(42,16)--(58,16)--(48.5,35)--(67.5,35)--(50,0)--(35.5,29)--(48.5,29)--(42,42)--(21,0)
  (26,94)--(55.5,35)--(85,94)
  (108.5,47)--(61.5,47)--(85,0)
  ;
%\node at (42,33.33) {\tiny $13$};  
\node at (50,21.33) {\tiny $16$};  
\node at (50,10.67) {\tiny $16$};  
\node at (58,28.67) {\footnotesize $19$};  
\node at (35.5,9.67) {$29$};  
\node at (67.5,11.67) {$35$};  
\node at (21,28) {$42$};  
\node at (85,31.33) {$47$};  
\node at (85,62.67) {$47$};  
\node at (26,59.33) {$52$};  
\node at (55.5,74.33) {$59$};  
\end{tikzpicture}(m)\quad
%\hfill
\begin{tikzpicture}[xscale=-.045,yscale=.0389]
\draw (0,60)--(36.5,133)--(120.5,133)--(159.5,55)--(132,0)--(30,0)--(0,60)--(73,60)--(66.5,47)--(60,60)--(30,0)
  (132,0)--(104.5,55)--(77,0)--(53.5,47)--(68.5,47)--(61,32)--(93,32)--(81.5,55)--(159.5,55)
  (76,32)--(67.5,49)--(84.5,49)--(76,32)
  (36.5,133)--(78.5,49)--(120.5,133)
  ;
\node at (73,52.67) {\scriptsize $11$};  
\node at (66.5,55.67) {\footnotesize $13$};  
\node at (60,51.33) {\footnotesize $13$};  
\node at (61,42) {$15$};  
\node at (68.5,37) {$15$};  
\node at (76,43.33) {$17$};  
\node at (84.5,37.67) {$17$};  
\node at (93,47.33) {$23$};  
\node at (77,21.67) {$32$};  
\node at (53.5,15.67) {$47$};  
\node at (104.5,18.33) {$55$};  
\node at (132,36.67) {$55$};  
\node at (30,40) {$60$};  
\node at (36.5,84.33) {$73$};  
\node at (120.5,81) {$78$};  
\node at (78.5,105) {$84$};  
\end{tikzpicture}(n)\quad
%\hfill
\\[2ex]
\begin{tikzpicture}[xscale=-.045,yscale=.0389]
\draw (0,82)--(49.5,181)--(163.5,181)--(216.5,75)--(179,0)--(41,0)--(0,82)--(99,82)--(90.5,65)--(92.5,65)--(82,44)--(126,44)--(110.5,75)--(216.5,75)
  (41,0)--(82,82)--(91.5,63)--(72.5,63)--(104,0)--(141.5,75)--(179,0)
  (103,44)--(91.5,67)--(114.5,67)--(103,44)
  (49.5,181)--(106.5,67)--(163.5,181)
  ;
\node at (110.5,69.67) {\scriptsize $8$};  
\node at (99,72) {\small $15$};  
\node at (90.5,76.33) {$17$};  
\node at (82,69.33) {$19$};  
\node at (82,56.67) {$19$};  
\node at (92.5,51) {$21$};  
\node at (103,59.33) {$23$};  
\node at (114.5,51.67) {$23$};  
\node at (126,64.67) {$31$};  
\node at (105,29.33) {$44$};  
\node at (72.5,21) {$63$};  
\node at (141.5,25) {$75$};  
\node at (179,50) {$75$};  
\node at (41,54.67) {$82$};  
\node at (49.5,115) {$99$};  
\node at (163.5,110.33) {$106$};  
\node at (106.5,143) {$114$};  
\end{tikzpicture}(o)\quad
%\hfill
\\[2ex]
\begin{tikzpicture}[xscale=-.09,yscale=.0779]
\draw   
  (53,71)--(86,71)--(78.5,56)--(80.5,56)--(71.5,38)--(109.5,38)--(96,65)--(133,65)
  (53,35)--(71,71)--(79,55)--(79.5,56)--(80,55)--(63,55)--(76.5,28)
  (104.5,28)--(123,65)--(133,45)
  (89.5,38)--(99.5,58)--(79.5,58)--(89.5,38)
  (81,81)--(92.5,58)--(104,81)
  ;
\draw[dotted]
  (53,28)--(133,28)--(133,81)--(53,81)--cycle
  ;
\node at (96,60.33) {$7$};  
\node at (86,62.33) {$13$};  
\node at (78.5,66) {$15$};  
\node at (71,60.33) {$16$};  
\node at (71.5,49.33) {$17$};  
\node at (80.5,44) {$18$};  
\node at (89.5,51.33) {$20$};  
\node at (99.5,44.67) {$20$};  
\node at (109.5,56) {$27$};  
\node at (90.5,33) {$38$};  
\node at (64,36) {$55$};
\node at (129.67,58.33) {$65$};  
\node at (122,39) {$65$};  
\node at (59,59) {$71$};  
\node at (68,76) {$86$};  
\node at (117,73) {$92$};  
\node at (92.5,73.33) {$99$};   
\end{tikzpicture}(p)\quad
%\hfill
\\[2ex]
\begin{tikzpicture}[xscale=.01125,yscale=.00974]
\draw
(0,208)--(125.5,459)--(408.5,459)--(533.5,209)--(429,0)--(104,0)--(0,208)--(251,208)--(229.5,165)--(208,208)--(104,0)
(429,0)--(324.5,209)--(300,160)--(349,160)--(269,0)--(186.5,165)--(240.5,165)--(213.5,111)--(324.5,111)--(296,168)--(304,168)--(283.5,209)--(533.5,209)
  (267.5,111)--(235,176)--(300,176)--(267.5,111)
  (125.5,459)--(267,176)--(408.5,459)
  ;
\node at (304,195.33) {\tiny $41$};  
\node at (229.5,193.67) {\tiny $43$};  
\node at (208,179.33) {\tiny $43$};  
\node at (324.5,176.33) {\scriptsize $49$};  
\node at (324.5,143.67) {\scriptsize $49$};  
\node at (240.5,129) {\footnotesize $54$};  
\node at (213.5,147) {\footnotesize $54$};  
\node at (296,130) {\footnotesize $57$};  
\node at (267.5,154.33) {$65$};  
\node at (269,74) {$111$};  
\node at (349,53.33) {$160$};  
\node at (186.5,55) {$165$};  
\node at (104,138.66) {$208$};  
\node at (429,139.33) {$209$};  
\node at (408.5,292.33) {$250$};  
\node at (125.5,291.67) {$251$};  
\node at (267,364.67) {$283$};  
\end{tikzpicture}(q)\quad
%\hfill
\\[2ex]
\begin{tikzpicture}[xscale=-.1125,yscale=.0974]
\draw (0,36.1)--(21.7,79.5)--(70.8,79.5)--(92.55,36.0)--(74.55,0)--(18.05,0)--(0,36.1)--(43.4,36.1)--(39.75,28.8)--(41.35,28.8)--(36.5,19.1)--(55.6,19.1)--(50.9,28.5)--(60.3,28.5)--(46.05,0)--(32.05,28.0)--(40.95,28.0)--(40.55,28.8)--(40.15,28.0)--(36.1,36.1)--(18.05,0)
  (74.55,0)--(56.55,36.0)--(52.8,28.5)--(49.05,36.0)--(92.55,36.0)
  (46.2,19.1)--(40.55,30.4)--(51.85,30.4)--(46.2,19.1)
  (21.7,79.5)--(46.25,30.4)--(70.8,79.5)
  ;
\node at (49.05,32.267) {\footnotesize $56$};  
\node at (43.4,32.3) {\footnotesize $57$};  
\node at (39.75,33.667) {$73$};  
\node at (52.8,33.5) {$75$};  
\node at (56.55,31) {$75$};  
\node at (36.1,30.7) {$81$};  
\node at (36.9,25.033) {$89$};  
\node at (55.6,25.367) {$94$};  
\node at (50.9,22.233) {$94$};  
\node at (41.35,22.333) {$97$};  
\node at (46.2,26.633) {$113$};  
\node at (46.05,12.733) {$191$};  
\node at (32.05,9.333) {$280$};  
\node at (60.3,9.833) {$285$};  
\node at (74.55,24) {$360$};  
\node at (18.05,24.067) {$361$};
\node at (21.7,50.567) {$434$};  
\node at (70.8,50.5) {$435$};  
\node at (46.25,63.133) {$491$};  
\end{tikzpicture}(r)\quad
%\hfill
\\[2ex]
\begin{tikzpicture}[xscale=.08182,yscale=.07085]
\draw (0,49.4)--(29.6,108.6)--(96.4,108.6)--(125.95,49.5)--(101.2,0)--(24.7,0)--(0,49.4)--(59.2,49.4)--(54.3,39.6)--(56.5,39.6)--(49.95,26.5)--(76.45,26.5)--(69.75,39.9)--(71.65,39.9)--(66.85,49.5)--(125.95,49.5) (24.7,0)--(49.4,49.4)--(54.85,38.5)--(55.4,39.6)--(55.95,38.5)--(43.95,38.5)--(63.2,0)--(82.2,38.0)--(70.7,38.0)--(76.45,49.5)--(101.2,0)
  (63.05,26.5)--(70.7,41.8)--(55.4,41.8)--(63.05,26.5)
  (29.6,108.6)--(63.0,41.8)--(96.4,108.6)
  ;
\node at (59.2,44.333) {\footnotesize $76$};  
\node at (66.85,44.367) {\footnotesize $77$};  
\node at (71.65,46.3) {\small $96$};  
\node at (54.3,46.133) {\small $98$};  
\node at (49.4,42.133) {\small $109$};  
\node at (76.45,34.167) {\small $115$};  
\node at (76.45,41.833) {\small $115$};  
\node at (49.95,34.5) {$120$};  
\node at (56.5,30.867) {$131$};  
\node at (69.75,30.967) {$134$};  
\node at (63.05,36.7) {$153$};  
\node at (63.2,17.667) {$265$};  
\node at (82.2,12.667) {$380$};  
\node at (43.95,12.833) {$385$};  
\node at (24.7,32.933) {$494$};  
\node at (101.2,33) {$495$};
\node at (96.35,69.133) {$591$};  
\node at (29.55,69.2) {$592$};  
\node at (62.95,86.333) {$668$};  
\end{tikzpicture}(s)\quad
%\hfill
\\[2ex]
\begin{tikzpicture}[xscale=-.045,yscale=.0390]
\draw (0,33)--(19,71)--(62,71)--(82.5,30)--(67.5,0)--(16.5,0)--(0,33)--(38,33)--(35.5,28)--(42.5,28)--(41.5,30)--(82.5,30)
  (19,71)--(40.5,28)--(62,71)
  (16.5,0)--(33,33)--(39,21)--(43.5,30)--(48,21)--(27,21)--(37.5,0)--(52.5,30)--(67.5,0)
  ;
\node at (40.5,56.67) {$43$};  
\node at (62,43.67) {$41$};  
\node at (18,45.67) {$38$};  
\node at (16.5,22) {$33$};  
\node at (67.5,20) {$30$};  
\node at (52.5,10) {$30$};  
\node at (27,7) {$21$};  
\node at (37.5,14) {$21$};  
\node at (33,25) {\footnotesize $12$};  
\node at (43.5,24) {\footnotesize $9$};  
\node at (48,27) {\footnotesize $9$};  
\node at (39,25.67) {\tiny $7$};  
\end{tikzpicture}(t)\quad
%\hfill
\\[2ex]
\begin{tikzpicture}[xscale=-.09,yscale=.0779]
\draw
  (89.5,115)--(115,115)--(110,125)
  (89.5,64)--(99.5,84)
  (120,43)--(115,53)
  (163,43)--(168,53)
  (190.5,58)--(180.5,78)
  (168,103)--(190.5,103)
  (155,125)--(144,103)
  (115,115)--(99.5,84)
  (115,115)--(132.5,80)--(144,103)
  (115,53)--(99.5,84)--(130.5,84)--(115,53)--(168,53)--(155,79)--(156,79)--(144,103)--(168,103)
  (168,53)--(180.5,78)--(155.5,78)--(168,103)--(180.5,78)
  (142,53)--(155.5,80)--(128.5,80)--(142,53)  
  ;
\draw[dotted]
  (89.5,43)--(190.5,43)--(190.5,125)--(89.5,125)--cycle
  ;  
\node at (130.5,81.33) {\scriptsize $4$};  
\node at (115,94.33) {$31$};  
\node at (115,73.67) {$31$};  
\node at (128.5,62) {$27$};  
\node at (142,71) {$27$};  
\node at (155,61.67) {$26$};  
\node at (168,69.67) {$25$};  
\node at (168,86.33) {$25$};  
\node at (156,95) {$24$};  
\node at (144,87.33) {$23$};  
\node at (141.5,48) {$53$};  
\node at (180,53) {$78$};  
\node at (101,55) {$84$};  
\node at (183,88) {$103$};  
\node at (100,120) {$115$}; 
\node at (99,98) {$115$};  
\node at (170,114) {$127$};  
\node at (132.5,110) {$150$};  
\end{tikzpicture}(u)\quad
%\hfill
\\[2ex]
\begin{tikzpicture}[xscale=-.09,yscale=.0779]
\draw
(0,39)--(19.5,0)--(92.5,0)--(109,33)--(92.5,66)--(60.5,66)--(45,35)--(21.5,82)--(0,39)--(43,39)--(41,35)--(76,35)--(58.5,0)--(39,39)--(19.5,0)
  (92.5,0)--(75.5,34)--(76.5,34)--(60.5,66)
  (109,33)--(76,33)--(92.5,66)  
  (-10,82)--(21.5,82)--(16.5,92)
  (0,39)--(-10,19)
  (19.5,0)--(24.5,-10)
  (92.5,0)--(87.5,-10)
  (109,33)--(119,13)
  (92.5,66)--(119,66)
  (60.5,66)--(73.5,92)
  ;
\draw[dotted]
  (-10,-10)--(119,-10)--(119,92)--(-10,92)--cycle
  ;
\node at (45,73) {$202$};  
\node at (93,79) {$171$};  
\node at (5,87) {$155$};  
\node at (-.5,61) {$155$};  
\node at (110.17,48.33) {$139$};  
\node at (2,6) {$112$};  
\node at (107,3) {$106$};  
\node at (56,-5) {$73$};  
\node at (21.5,53.33) {$43$};  
\node at (19.5,26) {$39$};  
\node at (39,13) {$39$};  
\node at (58.5,23.33) {$35$};  
\node at (75.5,11.33) {$34$};  
\node at (92.5,22) {$33$};  
\node at (92.5,44) {$33$};  
\node at (76.5,55.33) {$32$};  
\node at (60.5,45.33) {$31$};
\node at (41,37.67) {\scriptsize $4$};
\node at (43,36.33) {\scriptsize $4$};
\end{tikzpicture}(v)\quad
%\hfill
\\[2ex]
\begin{tikzpicture}[xscale=-.1,yscale=.0866]
\draw (0,26.7)--(13.35,0)--(54.45,0)--(67.65,26.4)--(51.8,58.1)--(15.7,58.1)--(0,26.7)--(31.4,26.7)--(29.05,22.0)--(38.15,22.0)--(34.15,14.0)--(30.15,22.0)--(29.6,20.9)--(26.7,26.7)--(13.35,0)
(54.45,0)--(41.25,26.4)--(38.15,20.2)--(44.35,20.2)--(34.25,0)--(23.8,20.9)--(30.7,20.9)--(27.25,14.0)--(41.25,14.0)--(37.7,21.1)--(38.6,21.1)--(35.95,26.4)--(67.65,26.4)
  (51.8,58.1)--(33.75,22.0)--(15.7,58.1)
  ;
\node at (33.75,46.07) {$361$};  
\node at (15.7,37.17) {$314$};  
\node at (13.35,17.8) {$267$};  
\node at (23.8,6.97) {$209$};  
\node at (34.25,9.33) {$140$};  
\node at (44.35,6.73) {$202$};
\node at (54.45,17.6) {$264$};
\node at (51.8,36.97) {$317$};  
\node at (34.15,19.33) {$80$};  
\node at (37.7,16.37) {$71$};
\node at (30.7,16.3) {$69$};  
\node at (27.25,18.6) {$69$};  
\node at (41.25,18.13) {\footnotesize $62$};  
\node at (41.25,22.27) {\footnotesize $62$};
\node at (26.7,22.83) {\footnotesize $58$};
\node at (38.6,24.63) {\scriptsize $53$};  

\node at (29.05,25.13) {\tiny $47$};  
\node at (31.5,23.57) {\tiny $47$};  

\node at (35.95,23.47) {\tiny $44$};  

\end{tikzpicture}(w)
\vspace{2ex}

\begin{table}
\caption{Parameters of the tilings from the appendix\label{tab:hexagons}}

\begin{tabular}{|c|c|p{83mm}|c|}
\hline
\multicolumn{4}{|c|}{Tilings showing that $\SHEX(n) \ge n-4$ or even $\SHEXp(n) \ge n-4$}\\
\hline
& $n$ & sizes of tiles (references of related tilings) & t-perfect \\
\hline
(a) & 9 & 1, 3, 3, 3, 4, 4, 5, 5, 6 & no\\
(b) & 10 & 1, 1, 4, 4, 4, 5, 6, 6, 7, 8 (cf.\ \cite[p.\ 1488, third tiling]{drapal2010}) & no\\
(c) & 11 & 2, 2, 3, 5, 5, 7, 7, 8, 8, 9, 11 (cf.\ \cite[Figure 1]{tutte1948}, \cite[Figure 4]{tuza1991}, \cite[Figure 8]{drapal2010}, \cite[perfect\underline{ }dissection\underline{ }size15\underline{ }595\underline{ }r5\underline{ }c3]{hamalainen}) & yes\\
(d) & 11 & 1, 3, 4, 4, 4, 7, 7, 8, 9, 9, 10 & no\\
(e) & 12 & 1, 3, 3, 4, 7, 7, 7, 10, 11, 12, 12, 13 & no\\
(f) & 12 & 1, 1, 1, 2, 7, 7, 7, 9, 10, 11, 12, 14 & no\\
(g) & 13 & 2, 2, 2, 4, 5, 9, 9, 11, 13, 14, 14, 15, 19 & no\\
(h) & 14 & 3, 3, 3, 5, 5, 6, 11, 14, 16, 16, 17, 20, 21, 26 & no\\
(i) & 14 & 2, 3, 3, 3, 6, 8, 11, 11, 14, 17, 19, 21, 21, 23 & no\\
(j) & 14 & 1, 4, 5, 5, 6, 6, 11, 11, 16, 17, 20, 20, 23, 24 (cf. \cite[p.\ 1488, last tiling]{drapal2010}, \cite[perfect\underline{ }dissection\underline{ }size18\underline{ }30413\underline{ }r6\underline{ }c3]{hamalainen}) & yes\\
(k) & 15 & 1, 2, 2, 3, 3, 4, 5, 8, 8, 12, 12, 13, 15, 16, 17 (cf. \cite[perfect\underline{ }dissection\underline{ }size18\underline{ }30411\underline{ }r3\underline{ }c6]{hamalainen})& yes\\
(l) & 16 & 3, 5, 8, 8, 9, 11, 11, 11, 14, 22, 25, 30, 34, 34, 38, 43 & no\\
(m) & 17 & 3, 3, 7, 10, 12, 13, 13, 16, 16, 19, 29, 35, 42, 47, 47, 52, 59 & yes\\
(n) & 18 & 2, 6, 11, 13, 13, 15, 15, 17, 17, 23, 32, 47, 55, 55, 60, 73, 78, 84 & yes\\
(o) & 19 & 2, 2, 8, 15, 17, 19, 19, 21, 23, 23, 31, 44, 63, 75, 75, 82, 99, 106, 114 & yes\\
(p) & 21 & 1, 1, 1, 2, 7, 13, 15, 16, 17, 18, 20, 20, 27, 38, 55, 65, 65, 71, 86, 92, 99 & no\\
(q) & 22 & 8, 8, 11, 32, 33, 41, 43, 43, 49, 49, 54, 54, 57, 65, 111, 160, 165, 208, 209, 250, 251, 283 & yes\\ 
(r) & 24 & 8, 8, 8, 16, 19, 56, 57, 73, 75, 75, 81, 89, 94, 94, 97, 113, 191, 280, 285, 360, 361, 434, 435, 491 & no\\
(s) & 25 & 11, 11, 11, 19, 19, 22, 76, 77, 96, 98, 109, 115, 115, 120, 131, 134, 153, 265, 380, 385, 494, 495, 591, 592, 668 & no\\
\hline\hline
\multicolumn{4}{|c|}{Tilings showing that $\SHEXp(n) \ge n-5$}\\
\hline
& $n$ & \multicolumn{2}{l|}{sizes of tiles (references of related tilings)}\\
\hline
(t) & 16 & \multicolumn{2}{p{100mm}|}{2, 2, 5, 5, 7, 9, 9, 12, 21, 21, 30, 30, 33, 38, 41, 43 (cf. \cite[perfect\underline{ }dissection\underline{ }size19\underline{ }30749\underline{ }r0\underline{ }c2]{hamalainen})}\\
(u) & 20 & \multicolumn{2}{p{100mm}|}{1, 1, 4, 23, 24, 25, 25, 26, 27, 27, 31, 31, 53, 78, 84, 103, 115, 115, 127, 150}\\
(v) & 21 & \multicolumn{2}{p{100mm}|}{1, 1, 4, 4, 31, 32, 33, 33, 34, 35, 39, 39, 43, 73, 106, 112, 139, 155, 155, 171, 202}\\
(w) & 23 & \multicolumn{2}{p{100mm}|}{9, 9, 11, 11, 44, 47, 47, 53, 58, 62, 62, 69, 69, 71, 80, 140, 202, 209, 264, 267, 314, 317, 361}\\
\hline
\end{tabular}
\end{table}

%%%%%%%%%%%%%%%%%%%%%%%%%%%%%%%%%%%%%%%%%%%%%%%%%%

\bibliographystyle{plain}

%%%%%%%%%%%%%%%%%%%%%%%%%%%%%%%%%%%%%%%%%%%%%%%%%%%%%%%

\end{document}